\documentclass[11pt,reqno,a4paper,twoside]{extarticle}

\usepackage[osf]{newtxtext}

\usepackage[left=3cm,right=3cm,top=2.2cm,bottom=2.2cm]{geometry}

\usepackage{graphicx, varioref, amsfonts}
\usepackage{amsthm,amssymb,amsmath,mathrsfs,mathtools,stmaryrd} 
\usepackage{dsfont}
\usepackage{upgreek} 
\usepackage{esint} 
\usepackage{breakurl}
\usepackage{empheq}
\usepackage{enumerate,enumitem}
\usepackage{hyperref}
\usepackage{cleveref}
\usepackage{soul}
\usepackage{cancel}
\usepackage{amscd}
\usepackage{srcltx}
\usepackage{ifthen}
\usepackage{float}
\usepackage{color}
\usepackage{comment}
\usepackage{chngcntr}
\usepackage{etoolbox}
\usepackage{fancyhdr}
\usepackage{breqn}
\usepackage{cite}
\usepackage[T1]{fontenc}

\usepackage{titlefoot}

\usepackage{authblk} 
\theoremstyle{plain}
\newtheorem{lemma}{Lemma}[section]

\newtheorem{proposition}{Proposition}[section]
\newtheorem{corollary}{Corollary}[section]
\newtheorem{theorem}{Theorem}[section]

\newtheorem{assumption}{Assumption}[section]

\theoremstyle{definition}

\newtheorem{definition}{Definition}[section]

\newtheorem{remark}{Remark}[section]

\newlist{todolist}{itemize}{2}
\setlist[todolist]{label=$\square$}
\usepackage{pifont}

\begin{document}

\title{Necessary and Sufficient Conditions for Optimal Control of Semilinear Stochastic Partial Differential Equations
}
\newcommand\shorttitle{Optimality Conditions for Controlled SPDEs}
\date{December 8, 2023}

\author{Wilhelm Stannat$^1$}
\author{Lukas Wessels$^2$}
\newcommand\authors{Wilhelm Stannat and Lukas Wessels}

\affil{\small ${}^1$Technische Universit\"at Berlin\\\small ${}^2$Georgia Institute of Technology}

\maketitle

\unmarkedfntext{\textit{Mathematics Subject Classification (2020) ---} 93E20, 49K45, 60H15, 49L20}


\unmarkedfntext{\textit{Keywords and phrases ---} stochastic maximum principle, Pontryagin maximum principle, dynamic programming, verification theorem, stochastic optimal control, Hamilton--Jacobi--Bellman equation}

\unmarkedfntext{\textit{Email}: \textbullet$\,$ stannat@math.tu-berlin.de $\,$\textbullet$\,$ wessels@gatech.edu}

\begin{abstract}
	Using a recently introduced representation of the second order adjoint state as the solution of a function-valued backward stochastic partial differential equation (SPDE), we calculate the viscosity super- and subdifferential of the value function evaluated along an optimal trajectory for controlled semilinear SPDEs. This establishes the well-known connection between Pontryagin's maximum principle and dynamic programming within the framework of viscosity solutions. As a corollary, we derive that the correction term in the stochastic Hamiltonian arising in non-smooth stochastic control problems is non-positive. These results directly lead us to a stochastic verification theorem for fully nonlinear Hamilton--Jacobi--Bellman equations in the framework of viscosity solutions.
\end{abstract}


\section{Introduction}
\noindent
The two main approaches to mathematical control theory are the dynamic programming approach introduced by Bellman, see \cite{bellman1957}, which studies the value function, and the approach via Pontryagin's maximum principle, see \cite{pontryagin1962}, which evolves around the adjoint process. Under smoothness assumptions, Pontryagin already established in the deterministic case that the adjoint state is given by the derivative of the value function evaluated along an optimal trajectory. In the stochastic case, the adjoint state $p$ has to be complemented by another process $q$, which arises in the construction of a solution for the backward stochastic differential equation when applying It\^o's representation theorem. Bismut extended Pontryagin's result for $p$ to this case, see \cite{bismut1978}, and Bensoussan identified $q$ in terms of the second order derivative of the value function, see \cite{bensoussan1982}. Since then, there have been various generalizations of these results, dispensing with the smoothness assumptions on the value function, both in the deterministic and stochastic setting, as well as in finite and infinite dimensions, see \cite{cannarsa1991,cannarsa1992,zhou1991,zhou19912}. However, the full generalization to the case of controlled stochastic partial differential equations (SPDEs) was still missing. The first part of this paper closes this gap using a recently obtained characterization of the second order adjoint state as the solution of a function-valued backward SPDE, see \cite{stannat2020}. In the second part of this paper, we show how this relationship between the adjoint states and the (generalized) derivatives of the value function naturally leads us to a non-smooth version of the classical verification theorem that was obtained in the finite dimensional case in \cite{zhou1997,gozzi2005,gozzi2010}.

More specifically, consider the following stochastic optimal control problem. For given $(s,x)\in [0,T] \times L^2(\Lambda)$, minimize
\begin{equation}\label{costfunctional}
	J(s,x;u) := \mathbb{E} \left [ \int_s^T \int_{\Lambda} l(x^u_t(\lambda),u_t) \mathrm{d}\lambda \mathrm{d}t + \int_{\Lambda} h(x^u_T(\lambda)) \mathrm{d}\lambda \right ]
\end{equation}
over $u\in \mathcal{U}_s$ subject to
\begin{equation}\label{stateequation}
	\begin{cases}
		\mathrm{d}x^u_t = [ \Delta x^u_t + b(x^u_t, u_t) ] \mathrm{d}t + \sigma(x^u_t,u_t) \mathrm{d}W_t,\quad t\in[s,T]\\
		x^u_s=x\in L^2(\Lambda).
	\end{cases}
\end{equation}
Here, $\Lambda \subset \mathbb R$ is a bounded interval, and $\mathcal{U}_s$ is a given set of admissible controls which take values in some control domain $U$. We study equation \eqref{stateequation} with homogeneous Dirichlet boundary conditions, the coefficients $b$ and $\sigma$ are Nemytskii operators, and $l:\mathbb R\times U \to \mathbb{R}$ and $h:\mathbb R \to \mathbb{R}$ are functions of at most quadratic growth in $x\in \mathbb{R}$. Furthermore, $W$ is a cylindrical Wiener process on a real separable Hilbert space $\Xi$.

The main object of study in the dynamic programming approach is the value function defined as
\begin{equation}
	V(s,x) := \inf_{u\in\mathcal{U}_s} J(s,x;u)
\end{equation}
which satisfies the dynamic programming principle
\begin{equation}
	V(s,x) = \inf_{u\in \mathcal{U}_s} \mathbb{E} \left [ \int_s^t \int_{\Lambda} l(x^u_r(\lambda),u_r) \mathrm{d}\lambda \mathrm{d}r + V(t,x^u_t) \right ], \qquad \forall t \in [s,T],
\end{equation}
see \cite[Theorem 2.24]{fabbri2017}. Formally, one can derive the associated Hamilton--Jacobi--Bellman equation
\begin{equation}\label{hjb}
	\begin{cases}
		V_s + \langle \Delta x, DV \rangle_{L^2(\Lambda)} + \inf_{u\in U} \mathcal{H}(x,u,DV,D^2 V) = 0\\
		V(T,x) = \int_{\Lambda} h(x(\lambda)) \mathrm{d}\lambda,
	\end{cases}
\end{equation}
where $\mathcal{H}: L^2(\Lambda) \times U \times L^2(\Lambda) \times L(L^2(\Lambda)) \to \mathbb{R}$ is given by
\begin{equation}
	\mathcal{H}(x,u,p,P) := \int_{\Lambda} l(x(\lambda),u) \mathrm{d}\lambda + \langle p, b(x,u) \rangle_{L^2(\Lambda)} + \frac12 \text{tr} \left ( \sigma(x,u)^{\ast} P \sigma(x,u) \right ).
\end{equation}

Pontryagin's maximum principle for the present control problem has been obtained in \cite{stannat2020}, where the authors introduced a novel characterization of the second order adjoint state $(P,Q)$ as the solution of a function-valued backward SPDE representing the kernel of the associated operator-valued process $P \in L^2([s,T]\times \Omega; L_2(L^2(\Lambda)))$. Utilizing this representation, we now investigate the relationship between the value function and the first and second order adjoint processes $(p,q)$ and $(P,Q)$, respectively. Let $(\bar x,\bar u)$ be an optimal pair of our control problem. If the value function is sufficiently smooth, a formal calculation leads to the following analogue of the classical relationship obtained by Bensoussan in \cite[Section 3.2]{bensoussan1982}:
\begin{equation}
	\begin{cases}
		D V(t,\bar x_t) = p_t\\
		D^2 V(t,\bar x_t) = P_t
	\end{cases}
\end{equation}
and
\begin{equation}
	\partial_s V(t,\bar x_t) = - \langle \Delta \bar x_t, p_t\rangle_{H^{-1}(\Lambda)\times H^1_0(\Lambda)} - \mathcal{H}(\bar x_t,\bar u_t,p_t,P_t).
\end{equation}
In the first part of the present paper we rigorously prove the corresponding statement of this result within the framework of viscosity differentials. More precisely, we prove for almost every $t\in [s,T]$
\begin{equation}\label{intro}
	[- \langle \Delta \bar x_t, p_t\rangle_{H^{-1}(\Lambda)\times H^1_0(\Lambda)} - \mathcal{G}(t,\bar x_t, \bar u_t), \infty)\times \{p_t\} \times \mathcal{S}_{\succeq P_t}(L^2(\Lambda)) \subset D^{1,2,+}_{t+,x} V(t, \bar x_t),
\end{equation}
$\mathbb{P}$--almost surely, where $\mathcal{S}_{\succeq P_t}(L^2(\Lambda))$ is the convex cone of symmetric, positive operators translated by $P_t$ (see equation \eqref{convexcone}), the derivative on the right-hand side is the parabolic viscosity superdifferential of the value function (see Definition \ref{parabolicviscositydifferential}), and $\mathcal{G}:[s,T]\times L^2(\Lambda)\times U \to \mathbb{R}$ is given by
\begin{equation}
	\mathcal{G}(t,x,u) := \mathcal{H}(x,u,p_t,P_t) + \text{tr} \left ( \sigma(x,u)^{\ast} \left [ q_t - P_t \sigma(\bar x_t,\bar u_t) \right ] \right ).
\end{equation}
This means in particular
\begin{multline}
	\limsup_{\tau\downarrow 0, z\to 0} \frac{1}{\tau + \| z \|_{L^2(\Lambda)}^2}\Big [ V(t+\tau,\bar x_t+z) - V(t,\bar x_t)\\
	+ ( \langle \Delta \bar x_t, p_t\rangle_{H^{-1}(\Lambda)\times H^1_0(\Lambda)} + \mathcal{G}(t,\bar x_t, \bar u_t))\tau - \langle p_t, z\rangle_{L^2(\Lambda)} - \frac12 \langle z, P_t z \rangle_{L^2(\Lambda)} \Big ] \leq 0.
\end{multline}
As another corollary of \eqref{intro}, we also calculate the separate space- and time-differentials of the value function, and derive that the correction term arising in non-smooth stochastic control problems is non-positive, i.e.,
\begin{equation}\label{intro2}
	\text{tr} \left ( \sigma(\bar x_t,\bar u_t) (q_t - P_t \sigma(\bar x_t,\bar u_t)) \right ) \leq 0.
\end{equation}
The proof of \eqref{intro2} in the finite dimensional case takes advantage of a correspondence between elements in the parabolic viscosity superdifferential and test functions for viscosity solutions, see Proposition \ref{testfunction}. In the framework of $B$-continuous viscosity solutions for equations involving an unbounded operator, test functions are required to satisfy additional regularity properties, and therefore this correspondence does not hold anymore, see Remark \ref{remarktestfunction}. In our case we resolve this issue with a careful analysis of an explicit classical approximation of the value-function along the state trajectory, exploiting its $H^1_0(\Lambda)$-regularity.

The two necessary conditions \eqref{intro} and \eqref{intro2} lead us naturally to the stochastic verification theorem, a non-smooth generalization of the following classical result, see e.g. \cite[Section 2.5]{fabbri2017}. Let $\mathcal{V}\in C^{1,2}([s,T]\times L^2(\Lambda))$ be a solution of the HJB equation \eqref{hjb}, and let $u^{\ast}$ be an admissible control with corresponding state $x^{\ast}$. If for almost every $t\in [s,T]$
\begin{equation}\label{argmin}
	u^{\ast}_t \in \text{argmin}_{u\in U} \mathcal{H}(x^{\ast}_t,u,D\mathcal{V}(t,x^{\ast}_t), D^2\mathcal{V}(t,x^{\ast}_t))
\end{equation}
$\mathbb{P}$--almost surely, then $u^{\ast}$ is an optimal control.

In general, the solution of the HJB equation is not sufficiently regular to apply the classical verification theorem, thus requiring a formulation that does not involve the derivatives of the value function. In the infinite dimensional case, recent progress has been made in this direction, see \cite{fabbri20172,federico2018}. However, both these works only consider semilinear HJB equations which excludes the case of control-dependent noise. We prove the following verification theorem for control-dependent noise where the derivatives of the value function in \eqref{argmin} are replaced by an element in the parabolic viscosity superdifferential (for the corresponding result in finite dimensions, see \cite{zhou1997,gozzi2005,gozzi2010}). Let $V$ be the value function, and let $u^{\ast}$ be an admissible control with corresponding state $x^{\ast}$. If there exist processes 
\begin{equation}
	(G_t,p_t,P_t) \in D^{1,2,+}_{t+,x} V(t, x^{\ast}_t)
\end{equation}
such that
\begin{equation}
	\mathbb{E} \left [ \int_s^T G_t + \langle \Delta x^{\ast}_t, p_t\rangle_{H^{-1}(\Lambda)\times H^1_0(\Lambda)} + \mathcal{H}( x^{\ast}_t, u^{\ast}_t,p_t,P_t) \mathrm{d}t \right ] \leq 0,
\end{equation}
then $u^{\ast}$ is an optimal control. For precise assumptions, see Theorem \ref{verification}. This result holds in a more general setting than the one introduced above. Since the verification theorem is of independent interest, we perform the proof under more general assumptions, see Section \ref{verificationtheorem}. In particular, we do not assume that the coefficients of the control problem are of Nemytskii-type and we do not restrict to one-dimensional space domains. The proof of this result in the finite dimensional case again takes advantage of Proposition \ref{testfunction}, and the infinite dimensional case involving an unbounded operator therefore requires similar arguments as the ones employed in the proof of \eqref{intro2}.

In order to characterize the value function as the unique viscosity solution of the HJB equation \eqref{hjb} in the sense of \cite[Definition 3.35]{fabbri2017}, additional assumptions are needed. We also prove a verification theorem in this setting in terms of a viscosity subsolution to the HJB equation, see Theorem \ref{ch4:verification2}. A similar result for the control of deterministic PDEs was obtained in \cite{fabbri2010}. The drawback of their result is that they assume the existence of a test function $\varphi$ which produces a maximum of $V-\varphi$ and whose derivatives coincide with the element $(G_t,p_t,P_t)$ in the parabolic viscosity superdifferential. We don't assume the existence of such a test function, but work instead with an explicit construction of a function $\phi$ (which is not a test function in the sense of \cite[Definition 3.32]{fabbri2017}) and exploit the higher regularity of the state trajectory. The verification theorem from \cite{fabbri2010} was also stated for the control of SPDEs in \cite{fabbri2006}. However, the proof relies on \cite[Chapter 5, Lemma 5.2]{yong1999}, which is incorrect as pointed out in \cite{federico2011}. Finally, let us mention the recent preprint \cite{chen2022}, in which the authors prove a stochastic verification theorem for controlled infinite dimensional systems. In this manuscript, the authors work under the assumption that the solution of the state equation takes values in the domain of the unbounded operators, which requires additional assumptions on the coefficients of the state equation. For a more detailed discussion, see Remark \ref{remarkiii}.

The paper is organized as follows. In Section \ref{notation}, we introduce the notation and state the assumptions that we are working under throughout the rest of the paper. Section \ref{necessaryoptimalityconditions} contains necessary optimality conditions complementing Pontryagin's maximum principle. In Subsection \ref{parabolicderivatives}, we consider the differential inclusion for the parabolic viscosity differential. Subsections \ref{sectionspace} and \ref{sectiontime} are concerned with the space- and time-differentials, respectively. In Subsection \ref{nonpositivity} we derive the non-positivity of the correction term. Section \ref{verificationtheorem} is devoted to the sufficient optimality condition, the stochastic verification theorem for SPDEs.

\section{Preliminaries}\label{notation}

\subsection{Assumptions}
Let $\Lambda\subset \mathbb{R}$ be a bounded interval, and $H^{\gamma}_0(\Lambda)$, $\gamma >0$, be the fractional Sobolev space of order $\gamma$ with Dirichlet boundary conditions. Furthermore, let $H^{-1}(\Lambda)$ denote the dual space of $H^1_0(\Lambda)$. We fix a finite time horizon $T>0$, an initial time $s\in [0,T)$, and solve the state equation \eqref{stateequation} in the variational setting on the Gelfand triple $H^1_0(\Lambda) \hookrightarrow L^2(\Lambda) \hookrightarrow H^{-1}(\Lambda)$, for details see \cite[Chapter 4]{liu2015}.

For real, separable Hilbert spaces $\mathcal{X}, \mathcal{Y}$, let $L(\mathcal{X},\mathcal{Y})$ denote the space of bounded linear operators from $\mathcal{X}$ to $\mathcal{Y}$, $L_2(\mathcal{X},\mathcal{Y}) \subset L(\mathcal{X},\mathcal{Y})$ denote the subspace of Hilbert-Schmidt operators, and $L_1(\mathcal{X},\mathcal{Y}) \subset L(\mathcal{X},\mathcal{Y})$ denote the subspace of nuclear operators. Let $L(\mathcal{X}) := L(\mathcal{X},\mathcal{X})$, $L_2(\mathcal{X}) := L_2(\mathcal{X},\mathcal{X})$, and $L_1(\mathcal{X}) := L_1(\mathcal{X},\mathcal{X})$. Furthermore, let $\mathcal{S}(\mathcal{X})$ denote the space of bounded, linear, symmetric operators on $\mathcal{X}$. By an abuse of notation, we are going to use the same notation for a function $P\in L^2(\Lambda^2)$ and the associated operator in $L_2(L^2(\Lambda))$ given by
\begin{equation}
	f \mapsto \int_{\Lambda} P(\cdot,\lambda) f(\lambda) \mathrm{d}\lambda,
\end{equation}
for $f\in L^2(\Lambda)$. Note that $\|P\|_{L^2(\Lambda^2)} = \| P \|_{L_2(L^2(\Lambda))}$.

We impose the following assumptions on the coefficients of the control problem:
\begin{assumption}\label{assumption}
	\begin{enumerate}
		\item Let $(\Omega^{\nu},\mathcal{F}^{\nu})$ be a standard measurable space (see \cite[Definition 1.11]{fabbri2017}), and let $(W^{\nu}_t)_{t\in [s,T]}$ be a cylindrical Wiener process on a probability space $(\Omega^{\nu},\mathcal{F}^{\nu},\mathbb{P}^{\nu})$ with values in some real separable Hilbert space $\Xi$ and $W^{\nu}_s =0$ $\mathbb{P}^{\nu}$-almost surely. Let $(\mathcal{F}^s_{\nu,t})_{t\in [s,T]}$ be the filtration generated by $(W^{\nu}_t)$ augmented by all $\mathbb{P}^{\nu}$-null sets. We call $\nu = (\Omega^{\nu}, \mathcal{F}^{\nu}, ( \mathcal{F}^s_{\nu,t})_{t\in [s,T]} ,\mathbb{P}^{\nu}, W^{\nu} )$ a reference probability space.
		\item Let $U$ be a non-empty subset of a separable Banach space $\mathcal{U}$, and let
		\begin{multline}
			\mathcal{U}_s^{\nu} := \Big \{ u:[s,T]\times \Omega \to U \Big | u \text{ is } (\mathcal{F}^s_{\nu,t})-\text{progressively measurable and}\\
			\sup_{t\in [s,T]} \mathbb{E} \left [ \| u_t \|_{\mathcal{U}}^k \right ] < \infty, \forall k\in \mathbb{N} \Big \}.
		\end{multline}
		The set of all admissible controls is given by
		\begin{equation}
			\mathcal{U}_s := \bigcup_{\nu} \mathcal{U}^{\nu}_s,
		\end{equation}
		where the union is taken over all reference probability spaces $\nu$.
		\item Let $l: \mathbb{R} \times U \to \mathbb{R}$, $h: \mathbb{R} \to \mathbb{R}$, $b: \mathbb{R} \times U \to \mathbb{R}$, and $\sigma : \mathbb{R} \times U \to L_2(\Xi, \mathbb{R})$ be twice continuously differentiable with respect to the first variable, $l,l_x,l_{xx} ,h,h_x,h_{xx}, b,b_x,b_{xx}, \sigma,\sigma_x,\sigma_{xx}$ be continuous, $l_{xx},h_{xx},b_x,b_{xx},\sigma_x,\sigma_{xx}$ be bounded, and $l_x,h_x, b, \sigma$ be bounded by $C(1+|x|+ \|u\|_{\mathcal{U}})$, and $l,h$ be bounded by $C(1+|x|^2+ \|u\|^2_{\mathcal{U}})$. By an abuse of notation, the associated Nemytskii operators on $L^2(\Lambda)$ are denoted by the same letters.
	\end{enumerate}
\end{assumption}
Under these assumptions, the value function satisfies the dynamic programming principle, see \cite[Theorem 2.24]{fabbri2017}.

Throughout the paper, we denote by
\begin{equation}
	\mathbb{E}^s_{\nu,t}[\, \cdot \, ] := \mathbb{E} [ \, \cdot \, | \mathcal{F}^s_{\nu,t} ]
\end{equation}
the conditional expectation on a given reference probability space $\nu$.

\subsection{Adjoint Equations}\label{adjointequations}

Throughout the paper, $(\bar x,\bar u)$ is going to be an optimal pair. We consider the first order adjoint equation
\begin{equation}\label{firstadjoint}
	\begin{cases}
		\mathrm{d}p_t = -\left [ \Delta p_t + b_{x}(\bar x_t,\bar u_t) p_t + \langle \sigma_{x}(\bar x_t,\bar u_t), q_t \rangle_{L_2(\Xi,\mathbb{R})} + l_{x}(\bar x_t,\bar u_t) \right ]\mathrm{d}t + q_t \mathrm{d}W_t\\
		p_T = h_{x}(\bar x_T).
	\end{cases}
\end{equation}
It is well known that the solution $(p,q)$, called the first order adjoint state, has the following regularity:
\begin{equation}
	p \in L^2 ([s,T]\times\Omega;H^1_0(\Lambda))\cap L^2(\Omega;C([s,T];L^2(\Lambda)))
\end{equation}
and
\begin{equation}
	q\in L^2([s,T]\times\Omega;L_2(\Xi,L^2(\Lambda))),
\end{equation}
see e.g. \cite[Theorem 4.1]{bensoussan1983}.

Furthermore, we consider the second order adjoint equation
\begin{equation}\label{secondadjoint}
	\begin{cases}
		\mathrm{d}P_t(\lambda,\mu) = - [ \Delta P_t(\lambda,\mu) + ( b_{x}(\bar x_t(\lambda),\bar u_t)) + b_{x}(\bar x_t(\mu),\bar u_t) ) P_t(\lambda,\mu)\\
		\qquad\qquad\qquad + \langle \sigma_{x}(\bar x_t(\lambda), \bar u_t), \sigma_{x}(\bar x_t(\mu), \bar u_t) \rangle_{L_2(\Xi,\mathbb{R})} P_t(\lambda,\mu) \\
		\qquad\qquad\qquad + \langle \sigma_{x} (\bar x_t(\lambda),\bar u_t) + \sigma_{x} (\bar x_t(\mu),\bar u_t), Q_t(\lambda,\mu)\rangle_{L_2(\Xi,\mathbb{R})}\\
		\qquad\qquad\qquad + \delta^{\ast}( l_{xx}(\bar x_t(\lambda), \bar u_t) ) + \delta^{\ast}( b_{xx}(\bar x_t(\lambda), \bar u_t) p_t(\lambda) )\\
		\qquad\qquad\qquad + \delta^{\ast}( \langle \sigma_{xx}(\bar x_t(\lambda),\bar u_t), q_t \rangle_{L_2(\Xi,\mathbb{R})} ) ] \mathrm{d}t + Q_t(\lambda,\mu) \mathrm{d}W_t\\
		P_T(\lambda,\mu) = \delta^{\ast}( h_{xx}(\bar x_T(\lambda)) ),
	\end{cases}
\end{equation}
where $\delta:H^1_0(\Lambda^2) \to L^2(\Lambda)$ is defined by $\delta(w)(\lambda) := w(\lambda,\lambda)$, and $\delta^{\ast} : L^2(\Lambda) \to H^{-1}(\Lambda^2)$ is its adjoint, i.e.
\begin{equation}
	\langle \delta^{\ast}(f) , w \rangle_{H^{-1}(\Lambda^2)\times H^1_0(\Lambda^2)} := \int_{\Lambda} f(\lambda) \delta(w)(\lambda) \mathrm{d}\lambda = \int_{\Lambda} f(\lambda) w(\lambda,\lambda) \mathrm{d}\lambda,
\end{equation}
for $f\in L^2(\Lambda)$, $w\in H^1_0(\Lambda^2)$. Since $P$ itself does not have an integral representation with integrands in $H^{-1}(\Lambda)$ and we need such an integral representation in order to apply It\^o's formula for variational processes, we need to introduce the mollified second order adjoint process as well. Consider the mollified second order adjoint equation
\begin{equation}\label{mollifiedsecondorderadjoint}
	\begin{cases}
		\mathrm{d}P^{\eta}_t(\lambda,\mu) = - [ \Delta P^{\eta}_t(\lambda,\mu) + ( b_{x}(\bar x_t(\lambda),\bar u_t) + b_{x}(\bar x_t(\mu),\bar u_t) ) P^{\eta}_t(\lambda,\mu)\\
		\qquad\qquad\qquad + \langle \sigma_{x}(\bar x_t(\lambda), \bar u_t), \sigma_{x}(\bar x_t(\mu), \bar u_t) \rangle_{L_2(\Xi,\mathbb{R})} P^{\eta}_t(\lambda,\mu) \\
		\qquad\qquad\qquad + \langle \sigma_{x} (\bar x_t(\lambda),\bar u_t) + \sigma_{x} (\bar x_t(\mu),\bar u_t), Q^{\eta}_t(\lambda,\mu)\rangle_{L_2(\Xi,\mathbb{R})}\\
		\qquad\qquad\qquad + \delta^{\ast}( l_{xx}(\bar x_t(\lambda), \bar u_t) ) + \delta^{\ast}( b_{xx}(\bar x_t(\lambda), \bar u_t) p_t(\lambda) )\\
		\qquad\qquad\qquad + \delta^{\ast}( \langle \sigma_{xx}(\bar x_t(\lambda),\bar u_t), q_t \rangle_{L_2(\Xi,\mathbb{R})} ) ] \mathrm{d}t + Q^{\eta}_t(\lambda,\mu) \mathrm{d}W_t\\
		P^{\eta}_T(\lambda,\mu) = h^{\eta}_{xx}(\lambda,\mu).
	\end{cases}
\end{equation}
where
\begin{equation}\label{meta}
	h^{\eta}_{xx}(\lambda,\mu) := 
	\frac 12 \left( h_{xx}(\bar x_T(\lambda)) + h_{xx}(\bar x_T(\mu))\right) \frac{1}{\sqrt{4\pi \eta}} \exp \left ( - \frac{|\lambda-\mu|^2}{4\eta} \right ) \in L^2(\Lambda^2).
\end{equation}
Note that
\begin{equation}
	\int_{\Lambda^2} h^{\eta}_{xx}(\lambda,\mu) w(\lambda,\mu) \mathrm{d}\lambda \mathrm{d}\mu \to \int_{\Lambda} h_{xx}(\bar x_T(\lambda)) w(\lambda,\lambda) \mathrm{d}\lambda
\end{equation}
$\mathbb{P}$-almost surely, for all $w\in H^1_0(\Lambda^2)$. Therefore, $h^{\eta}_{xx}$ approximates the terminal condition of \eqref{secondadjoint}. Furthermore, $h^{\eta}_{xx}$ is symmetric and due to the structure of equation \eqref{meta}, $P^{\eta}$ inherits this property.

We have the following result about existence, uniqueness and the convergence of solutions of equations \eqref{secondadjoint} and \eqref{mollifiedsecondorderadjoint}, see \cite{stannat2020}.
\begin{theorem}
	There exists a unique solution $(P,Q)$ of equation \eqref{secondadjoint} satisfying
	\begin{equation}
		P \in L^2([s,T]\times\Omega;L^2(\Lambda^2)) \cap L^2(\Omega;C([s,T];H^{-1}(\Lambda^2)))
	\end{equation}
	and
	\begin{equation}
		Q\in L^2([s,T]\times\Omega;L_2(\Xi,H^{-1}(\Lambda^2))).
	\end{equation}
	Furthermore, there exists a unique solution $(P^{\eta},Q^{\eta})$ of equation \eqref{mollifiedsecondorderadjoint} satisfying
	\begin{equation}
		P^{\eta} \in L^2([s,T]\times\Omega; H^1_0(\Lambda^2)) \cap L^2(\Omega; C([s,T];L^2(\Lambda)))
	\end{equation}
	and
	\begin{equation}
		Q^{\eta} \in L^2([s,T]\times\Omega;L_2(\Xi,L^2(\Lambda^2))).
	\end{equation}
	Finally, it holds
	\begin{align}
		P^{\eta} &\to P \quad \text{in}\; L^2([0,T]\times\Omega; L^2(\Lambda^2)),\\
		Q^{\eta} &\to Q \quad\text{in}\; L^2([0,T]\times\Omega; L_2(\Xi,H^{-1}(\Lambda^2))).
	\end{align}
\end{theorem}
The convergence of the mollified second order adjoinst state $(P^{\eta},Q^{\eta})$ to the second order adjoint state $(P,Q)$ follows from the a priori estimate
\begin{equation}
	\begin{split}
		&\mathbb{E} \left [ \int_s^T \| P^{\eta}_r - P_r \|_{L^2(\Lambda^2)}^2 \mathrm{d}r + \int_s^T \| Q^{\eta}_r - Q_r \|_{L_2(\Xi,H^{-1}(\Lambda^2))}^2 \mathrm{d}r \right ]\\
		&\leq C \mathbb{E} \left [ \| h^{\eta}_{xx} - \delta^{\ast}(h_{xx}(\bar x_T)) \|_{H^{-1}(\Lambda^2)}^2 \right ],
	\end{split}
\end{equation}
which can be proven analogously to the a priori bound \cite[equation (6.8)]{stannat2020}.

\begin{remark}\label{deltaremark}
	Note that the restriction to operators of Nemytskii-type and to one-dimensional space domains is applied in order to characterize the second order adjoint state as the solution of equation \eqref{secondadjoint}. More precisely, the restriction to one space dimension is applied in order for $\delta$ to be a continuous operator from $H^1_0(\Lambda^2) \to L^2(\Lambda)$. In higher space dimmensions, one needs more regularity than $H^1_0(\Lambda^2)$ for $\delta$ to be continuous. This can be achieved by imposing additional regularity assumptions on the noise coefficient $\sigma$, see also \cite[Remark 4.3]{stannat2020}.
\end{remark}

\section{Necessary Optimality Conditions}\label{necessaryoptimalityconditions}

\subsection{Parabolic Derivatives}\label{parabolicderivatives}

In this section, we are going to prove a relationship between the parabolic viscosity super- and subdifferentials of the value function on the one hand and the first and second order adjoint state on the other hand.

First, let us first recall the definition of the parabolic viscosity super- and subdifferential.
\begin{definition}\label{parabolicviscositydifferential}
	For $v\in C([s,T]\times L^2(\Lambda))$ the parabolic viscosity superdifferential of $v$ at $(t,x)\in [s,T)\times L^2(\Lambda)$ is the set
	\begin{multline}
		D^{1,2,+}_{t+,x} v(t,x) := \left \{ (G,p,P) \in \mathbb{R} \times L^2(\Lambda) \times \mathcal{S}(L^2(\Lambda)) \middle | \limsup_{\tau\downarrow 0, z\to 0} \frac{1}{\tau + \| z \|_{L^2(\Lambda)}^2}\right.\\
		\left. \vphantom{\frac{1}{|t-\hat t| + \| x-\hat x\|_{L^2(\Lambda)}^2}} \left [ v(t+\tau,x+z) - v(t, x) - G\tau - \langle p, z \rangle_{L^2(\Lambda)} - \frac12 \langle z, Pz \rangle_{L^2(\Lambda)} \right ] \leq 0 \right  \}.
	\end{multline}
	The parabolic viscosity subdifferential $D^{1,2,-}_{t+,x} v$ is defined analogously with the $\limsup$ replaced by $\liminf$ and the $\leq$ replaced by $\geq$.
\end{definition}
Now, we are ready to state the main result of this section.
\begin{theorem}\label{spacetime}
	Let $(\bar x,\bar u)$ be an optimal pair of the control problem \eqref{costfunctional} and \eqref{stateequation}, $(p,q)$ and $(P,Q)$ be the first and second order adjoint states, respectively, and $V$ be the value function. Then it holds for almost every $t\in [s,T]$
	\begin{equation}\label{4.84}
		[- \langle \Delta\bar x_t, p_t\rangle_{H^{-1}(\Lambda)\times H^1_0(\Lambda)} - \mathcal{G}(t,\bar x_t, \bar u_t), \infty)\times \{p_t\} \times \mathcal{S}_{\succeq P_t}(L^2(\Lambda)) \subset D^{1,2,+}_{t+,x} V(t, \bar x_t),
	\end{equation}
	$\mathbb{P}$--almost surely, and for almost every $t\in [s,T]$
	\begin{equation}\label{4.85}
		D^{1,2,-}_{t+,x} V(t,\bar x_t) \subset (-\infty, - \langle \Delta \bar x_t, p_t\rangle_{H^{-1}(\Lambda)\times H^1_0(\Lambda)} - \mathcal{G}(t,\bar x_t,\bar u_t)] \times \{p_t \} \times \mathcal{S}_{\preceq P_t}(L^2(\Lambda)),
	\end{equation}
	$\mathbb{P}$--almost surely, where
	\begin{equation}\label{convexcone}
		\mathcal{S}_{\succeq P_t}(L^2(\Lambda)) := \{ S\in \mathcal{S}(L^2(\Lambda)): S-P_t \text{ is a positive operator}\},
	\end{equation}
	and $\mathcal{S}_{\preceq P_t}(L^2(\Lambda))$ is defined analogously.
\end{theorem}

\begin{remark}
	Equation \eqref{4.84} in particular implies that the parabolic viscosity superdifferential is not empty.
\end{remark}

First, we discuss several lemmata that are needed in the proof of Theorem \ref{spacetime}. We suggest that the reader skip directly to the proof of Theorem \ref{spacetime} in Section \ref{proof} and refer to the lemmata as needed.

\subsubsection{Variational Equation}\label{variationalequation}

We begin by introducing the variational equation and deriving a priori bounds as well as regularity results for the solution.

\begin{lemma}\label{aprioriy}
	Let $\tau \geq 0$ and $z\in L^2(\Lambda)$, and let
	\begin{equation}
		\begin{cases}
			\mathrm{d} x^{\tau,z}_r = \left [ \Delta x^{\tau,z}_r + b( x^{\tau,z}_r, \bar u_r) \right ] \mathrm{d}r + \sigma( x^{\tau,z}_r, \bar u_r) \mathrm{d}W_r, \quad r\in [t+\tau,T]\\
			x^{\tau,z}_{t+\tau} = z+\bar x_t\in L^2(\Lambda).
		\end{cases}
	\end{equation}
	Define $y^{\tau,z}_r := x^{\tau,z}_r - \bar x_r$, i.e.
	\begin{equation}\label{variationalprocess}
		\begin{cases}
			\mathrm{d} y^{\tau,z}_r = \left [ \Delta y^{\tau,z}_r + b( x^{\tau,z}_r, \bar u_r) - b( \bar x_r, \bar u_r) \right ] \mathrm{d}r\\
			\qquad\qquad\qquad\qquad\qquad + \left [ \sigma(x^{\tau,z}_r, \bar u_r) - \sigma(\bar x_r, \bar u_r) \right ] \mathrm{d}W_r,\quad r\in [t+\tau,T]\\
			y^{\tau,z}_{t+\tau} = z+\bar x_t -\bar x_{t+\tau}\in L^2(\Lambda).
		\end{cases}
	\end{equation}
	Then, for any $k\in \mathbb{N}$ it holds for almost every $t\in [s,T]$
	\begin{equation}\label{aprioriy2}
		\mathbb{E}^s_{\nu,t} \left [ \left ( \int_{t+\tau}^T \| y^{\tau,z}_r \|_{H^1_0(\Lambda)}^{2} \mathrm{d}r \right )^k + \sup_{t+\tau \leq r \leq T} \left \| y^{\tau,z}_r \right \|_{L^2(\Lambda)}^{2k} \right ] \leq C \left ( \tau^k + \| z \|_{L^2(\Lambda)}^{2k} \right )
	\end{equation}
	$\mathbb{P}$--almost surely.
\end{lemma}

\begin{proof}
	Using standard a priori estimates, we obtain
	\begin{equation}
		\begin{split}
			&\mathbb{E}^s_{\nu,t} \left [ \left ( \int_{t+\tau}^T \| y^{\tau,z}_r \|_{H^1_0(\Lambda)}^{2} \mathrm{d}r \right )^k + \sup_{t+\tau \leq r \leq T} \left \| y^{\tau,z}_r \right \|_{L^2(\Lambda)}^{2k} \right ]\\
			&\leq C \mathbb{E}^s_{\nu,t} \left [ \left \| z+\bar x_t-\bar x_{t+\tau} \right \|^{2k}_{L^2(\Lambda)} \right ]\\
			&\leq C \left ( \left \| z \right \|^{2k}_{L^2(\Lambda)} + \mathbb{E}^s_{\nu,t} \left [ \left \| \bar x_{t+\tau} -\bar x_t \right \|^{2k}_{L^2(\Lambda)} \right ] \right ).
		\end{split}
	\end{equation}
	The claim follows from the fact that
	\begin{equation}\label{timeregularityx}
		\mathbb{E}^s_{\nu,t} \left [ \left \| \bar x_{t+\tau} -\bar x_t \right \|^{2k}_{L^2(\Lambda)} \right ] \leq C \tau^k,
	\end{equation}
	for every $k\in\mathbb{N}$.
\end{proof}

Next, we derive a Taylor expansion for the variational process $y^{\tau,z}$.
\begin{lemma}\label{variationalprocesstaylor}
	The variational process $y^{\tau,z}$ given by \eqref{variationalprocess} satisfies the equations
	\begin{equation}\label{zerothorder1}
		\begin{cases}
			\mathrm{d} y^{\tau,z}_r = \left [ \Delta y^{\tau,z}_r + b_x(\bar x_r, \bar u_r) y^{\tau,z}_r + \varphi^{1}_r \right ] \mathrm{d}r+ \left [ \sigma_x(\bar x_r, \bar u_r) y^{\tau,z}_r + \psi^{1}_r \right ] \mathrm{d}W_r,\quad r\in [t+\tau,T]\\
			y^{\tau,z}_{t+\tau} = z+\bar x_t -\bar x_{t+\tau} \in L^2(\Lambda),
		\end{cases}
	\end{equation}
	where
	\begin{equation}
		\begin{split}
			\varphi^{1}_r &:= \int_0^1 \left [ b_x(\bar x_r + \theta y^{\tau,z}_r, \bar u_r) - b_x(\bar x_r, \bar u_r) \right ] y^{\tau,z}_r \mathrm{d}\theta\\
			\psi^{1}_r &:= \int_0^1 \left [ \sigma_x(\bar x_r + \theta y^{\tau,z}_r, \bar u_r) - \sigma_x(\bar x_r, \bar u_r) \right ] y^{\tau,z}_r \mathrm{d}\theta,
		\end{split}
	\end{equation}
	and
	\begin{equation}\label{ysecondorder}
		\begin{cases}
			\mathrm{d} y^{\tau,z}_r = \left [ \Delta y^{\tau,z}_r + b_x(\bar x_r, \bar u_r) y^{\tau,z}_r + \frac12 b_{xx}(\bar x_r, \bar u_r) y^{\tau,z}_r y^{\tau,z}_r + \varphi^{2}_r \right ] \mathrm{d}r\\
			\qquad\qquad + \left [ \sigma_x(\bar x_r, \bar u_r) y^{\tau,z}_r + \frac12 \sigma_{xx}(\bar x_r, \bar u_r) y^{\tau,z}_r y^{\tau,z}_r + \psi^{2}_r \right ] \mathrm{d}W_r,\quad r\in [t+\tau,T]\\
			y^{\tau,z}_{t+\tau} = z+\bar x_t -\bar x_{t+\tau}\in L^2(\Lambda),
		\end{cases}
	\end{equation}
	where
	\begin{equation}
		\begin{split}
			\varphi^{2}_r &:= \int_0^1 (1-\theta) \left [ b_{xx}(\bar x_r + \theta y^{\tau,z}_r, \bar u_r) - b_{xx}(\bar x_r, \bar u_r) \right ] y^{\tau,z}_r y^{\tau,z}_r \mathrm{d}\theta\\
			\psi^{2}_r &:= \int_0^1 (1-\theta) \left [ \sigma_{xx}(\bar x_r + \theta y^{\tau,z}_r, \bar u_r) - \sigma_{xx}(\bar x_r, \bar u_r) \right ] y^{\tau,z}_r y^{\tau,z}_r \mathrm{d}\theta.
		\end{split}
	\end{equation}
	The remainder terms satisfy for every $k\in\mathbb{N}$, for almost every $t\in [s,T]$
	\begin{equation}\label{order1}
		\begin{split}
			&\mathbb{E}^s_{\nu,t} \left [ \int_{t+\tau}^T  \| \varphi^{1}_r \|_{L^2(\Lambda)}^{2k} \mathrm{d}r \right ] = o\left ( \tau^k + \| z \|_{L^2(\Lambda)}^{2k} \right ),\\
			&\mathbb{E}^s_{\nu,t} \left [ \int_{t+\tau}^T  \| \psi^{1}_r \|_{L_2(\Xi,L^2(\Lambda))}^{2k} \mathrm{d}r \right ] = o\left ( \tau^k + \| z \|_{L^2(\Lambda)}^{2k} \right ),
		\end{split}
	\end{equation}
	$\mathbb{P}$--almost surely, and for almost every $t\in [s,T]$
	\begin{equation}\label{order2}
		\begin{split}
			&\mathbb{E}^s_{\nu,t} \left [ \int_{t+\tau}^T  \| \varphi^{2}_r \|_{L^2(\Lambda)}^k \mathrm{d}r \right ] = o\left ( \tau^k + \| z \|_{L^2(\Lambda)}^{2k} \right ),\\
			&\mathbb{E}^s_{\nu,t} \left [ \int_{t+\tau}^T  \| \psi^{2}_r \|_{L_2(\Xi,L^2(\Lambda))}^k \mathrm{d}r \right ] = o\left ( \tau^k + \| z \|_{L^2(\Lambda)}^{2k} \right ),
		\end{split}
	\end{equation}
	$\mathbb{P}$--almost surely.
\end{lemma}

\begin{proof}
	The equations follow from the original equation for $y^{\tau,z}$ and Taylor's theorem for the G\^ateaux derivative (see \cite[Section 4.6]{zeidler1986}).
	
	Now let us prove the first asymptotic in \eqref{order1}. By Lipschitz continuity of the derivative of $b$, we have
	\begin{equation}
		\begin{split}
			& \mathbb{E}^s_{\nu,t} \left [ \int_{t+\tau}^T \left \| \int_0^1 \left [ b_x(\bar x_r + \theta y^{\tau,z}_r, \bar u_r) - b_x(\bar x_r, \bar u_r) \right ] y^{\tau,z}_r \mathrm{d}\theta \right \|_{L^2(\Lambda)}^{2k} \mathrm{d}r \right ]\\
			&\leq \mathbb{E}^s_{\nu,t} \left [ \int_{t+\tau}^T \int_0^1 \theta^{2k} \left \| y^{\tau,z}_r \right \|_{L^2(\Lambda)}^{4k} \mathrm{d}\theta \mathrm{d}r \right ] \\
			&\leq \mathbb{E}^s_{\nu,t} \left [ \int_{t+\tau}^T \left \| y^{\tau,z}_r \right \|_{L^2(\Lambda)}^{4k} \mathrm{d}r \right ]\\
			&\leq C \left ( \tau^{2k} + \| z \|_{L^2(\Lambda)}^{4k} \right ),
		\end{split}
	\end{equation}
	where we used Lemma \ref{aprioriy} in the last step. The remaining estimates follow analogously using the Lipschitz continuity of the first derivative of $\sigma$ and the Lipschitz continuity of the second derivatives of $b$ and $\sigma$.
\end{proof}
The following higher regularity of the variational process at the terminal time is needed in order to extract convergent subsequences as $\tau + \| z \|_{L^2(\Lambda)}^2$ tends to zero.
\begin{lemma}\label{regularityterminalcondition}
	Let $y^{\tau,z}$ be the variational process given by equation \eqref{zerothorder1}. Then, for any $\gamma \in (0,1/4)$, we have for almost every $t\in [s,T]$
	\begin{equation}
		\mathbb{E}^s_{\nu,t} \left [ \left \| y^{\tau,z}_T \right \|_{H^{\gamma}_0(\Lambda)}^2 \right ] \leq C \left ( \tau + \| z \|_{L^2(\Lambda)}^2 \right )
	\end{equation}
	$\mathbb{P}$--almost surely.
\end{lemma}

\begin{proof}
	By Lemma \ref{variationalprocesstaylor} and Duhamel's formula, we have
	\begin{equation}
		\begin{split}
			y^{\tau,z}_T = &S_{T-t-\tau} (z+\bar x_t-\bar x_{t+\tau}) + \int_{t+\tau}^T S_{r-t-\tau} \left (b_x(\bar x_r, \bar u_r) y^{\tau,z}_r + \varphi^{1}_r \right ) \mathrm{d}r\\
			&+ \int_{t+\tau}^T S_{r-t-\tau} \left (\sigma_x(\bar x_r, \bar u_r) y^{\tau,z}_r + \psi^{1}_r \right ) \mathrm{d}W_r,
		\end{split}
	\end{equation}
	where $(S_r)_{r\geq 0}$ denotes the heat semigroup. Now, let us estimate the $H^{\gamma}_0(\Lambda)$-norm. For the term involving the initial condition, we have
	\begin{equation}
		\begin{split}
			\mathbb{E}^s_{\nu,t} \left [ \| S_{T-t-\tau}(z+\bar x_t -\bar x_{t+\tau}) \|_{H^{\gamma}_0(\Lambda)}^2 \right ] &\leq C \mathbb{E}^s_{\nu,t} \left [ \|z+\bar x_t- \bar x_{t+\tau} \|_{L^2(\Lambda)}^2 \right ]\\
			&\leq C \left ( \left \| z \right \|^{2}_{L^2(\Lambda)} + \mathbb{E}^s_{\nu,t} \left [ \left \| \bar x_{t+\tau} - \bar x_t \right \|^{2}_{L^2(\Lambda)} \right ] \right ).
		\end{split}
	\end{equation}
	Since
	\begin{equation}
		\mathbb{E}^s_{\nu,t} \left [ \left \| \bar x_{t+\tau} - \bar x_t \right \|^{2}_{L^2(\Lambda)} \right ] \leq C \tau,
	\end{equation}
	this term satisfies the required bound. Now let us get to the first integral. We have
	\begin{equation}
		\begin{split}
			&\mathbb{E}^s_{\nu,t} \left [ \left \| \int_{t+\tau}^T S_{r-t-\tau} \left (b_x(\bar x_r, \bar u_r) y^{\tau,z}_r \right ) \mathrm{d}r \right \|_{H^{\gamma}_0(\Lambda)}^2\right ]\\
			&\leq C \mathbb{E}^s_{\nu,t} \left [ \int_{t+\tau}^T \frac{1}{(r-t-\tau)^{2\gamma}} \left \| b_x(\bar x_r, \bar u_r) y^{\tau,z}_r \right \|_{L^2(\Lambda)}^2 \mathrm{d}r \right ]\\
			&\leq C \sup_{t+\tau\leq r\leq T} \mathbb{E}^s_{\nu,t} \left [ \left \| y^{\tau,z}_r \right \|_{L^2(\Lambda)}^2 \right ] \int_{t+\tau}^T \frac{1}{(r-t-\tau)^{2\gamma}} \mathrm{d}r.
		\end{split}
	\end{equation}
	The required bound now follows from Lemma \ref{aprioriy}. For the second part of the first integral, we have
	\begin{equation}
		\begin{split}
			\mathbb{E}^s_{\nu,t} \left [ \left \| \int_{t+\tau}^T S_{r-t-\tau} \varphi^{1}_r \mathrm{d}r \right \|_{H^{\gamma}_0(\Lambda)}^2\right ] &\leq C \mathbb{E}^s_{\nu,t} \left [ \int_{t+\tau}^T \frac{\left \| \varphi^{1}_r \right \|_{L^2(\Lambda)}^2}{(r-t-\tau)^{2\gamma}} \mathrm{d}r \right ]\\
			&\leq C \mathbb{E}^s_{\nu,t} \left [ \int_{t+\tau}^T \left \| \varphi^{1}_r \right \|_{L^2(\Lambda)}^4 \mathrm{d}r \right ]^{\frac12} \mathbb{E}^s_{\nu,t} \left [ \int_{t+\tau}^T \frac{1}{(r-t-\tau)^{4\gamma}} \mathrm{d}r \right ]^{\frac12}.
		\end{split}
	\end{equation}
	The required bound for this term follows from Lemma \ref{variationalprocesstaylor}. For the stochastic integral, we have
	\begin{equation}
		\begin{split}
			&\mathbb{E}^s_{\nu,t} \left [ \left \| \int_{t+\tau}^T S_{r-t-\tau} \left (\sigma_x(\bar x_r, \bar u_r) y^{\tau,z}_r + \psi^{1}_r \right ) \mathrm{d}W_r \right \|_{H^{\gamma}_0(\Lambda)}^2\right ]\\
			&= \mathbb{E}^s_{\nu,t} \left [ \int_{t+\tau}^T \left \| S_{r-t-\tau} \left (\sigma_x(\bar x_r, \bar u_r) y^{\tau,z}_r + \psi^{1}_r \right ) \right \|_{L_2(\Xi,H^{\gamma}_0(\Lambda))}^2 \mathrm{d}r \right ]\\
			&\leq \mathbb{E}^s_{\nu,t} \left [ \int_{t+\tau}^T \frac{C}{(r-t-\tau)^{2\gamma}} \left \| \sigma_x(\bar x_r, \bar u_r) y^{\tau,z}_r + \psi^{1}_r \right \|_{L_2(\Xi,L^2(\Lambda))}^2 \mathrm{d}r \right ].
		\end{split}
	\end{equation}
	Using the same argument as above yields the claim.
\end{proof}

\subsubsection{Duality Relations}\label{dualityrelations}

Now, we discuss the duality relations for the first and second order adjoint states, respectively.

\begin{lemma}\label{duality11}
	Let $y^{\tau,z}$ be the variational process given by \eqref{ysecondorder} and let $p$ be the first order adjoint state. Then it holds for almost every $t\in [s,T]$
	\begin{equation}
		\begin{split}
			&\mathbb{E}^s_{\nu,t} \left [ \int_{t+\tau}^T \langle l_x(\bar x_r, \bar u_r), y^{\tau,z}_r \rangle_{L^2(\Lambda)} \mathrm{d}r + \langle h_x(\bar x_T), y^{\tau,z}_T \rangle_{L^2(\Lambda)} \right ]\\
			&= \mathbb{E}^s_{\nu,t} \left [ \langle p_{t+\tau}, y^{\tau,z}_{t+\tau} \rangle_{L^2(\Lambda)} \right ]\\
			&\quad+ \mathbb{E}^s_{\nu,t} \left [ \frac12 \int_{t+\tau}^T \langle p_r, b_{xx}(\bar x_r, \bar u_r) y^{\tau,z}_r y^{\tau,z}_r \rangle_{L^2(\Lambda)} + \langle q_r, \sigma_{xx}(\bar x_r, \bar u_r) y^{\tau,z}_r y^{\tau,z}_r \rangle_{L_2(\Xi,L^2(\Lambda))} \mathrm{d}r \right ]\\
			&\quad+ \mathbb{E}^s_{\nu,t} \left [ \int_{t+\tau}^T \langle p_r , \varphi^{2}_r \rangle_{L^2(\Lambda)} + \langle q_r, \psi^{2}_r \rangle_{L_2(\Xi,L^2(\Lambda))} \mathrm{d}r \right ],
		\end{split}
	\end{equation}
	$\mathbb{P}$--almost surely, and for almost every $t\in [s,T]$
	\begin{equation}\label{remainderp}
		\mathbb{E}^s_{\nu,t} \left [ \int_{t+\tau}^T \langle p_r , \varphi^{2}_r \rangle_{L^2(\Lambda)} + \langle q_r, \psi^{2}_r \rangle_{L_2(\Xi,L^2(\Lambda))} \mathrm{d}r \right ] = o\left ( \tau + \| z \|_{L^2(\Lambda)}^2 \right )
	\end{equation}
	$\mathbb{P}$--almost surely.
\end{lemma}

\begin{proof}
	Applying It\^o's formula for variational solutions to SPDEs (see \cite[Lemma 2.15]{pardoux2021} or \cite[Section 3]{krylov2013}) to $\langle y^{\tau,z}_r, p_r \rangle_{L^2(\Lambda)}$ and taking the conditional expectation yields the claim.
	
	For the remainder term estimate, we observe
	\begin{equation}
		\begin{split}
			&\mathbb{E}^s_{\nu,t} \left [ \int_{t+\tau}^T \langle p_r , \varphi^{2}_r \rangle_{L^2(\Lambda)} + \langle q_r, \psi^{2}_r \rangle_{L_2(\Xi,L^2(\Lambda))} \mathrm{d}r \right ]\\
			&\leq \mathbb{E}^s_{\nu,t} \left [ \int_{t+\tau}^T \| p_r \|_{L^2(\Lambda)} \| \varphi^{2}_r \|_{L^2(\Lambda)} + \| q_r \|_{L_2(\Xi,L^2(\Lambda))} \| \psi^{2}_r \|_{L_2(\Xi,L^2(\Lambda))} \mathrm{d}r \right ]\\
			&\leq \mathbb{E}^s_{\nu,t} \left [ \int_{t+\tau}^T \| p_r \|_{L^2(\Lambda)}^2 \mathrm{d}r \right ]^{\frac12} \mathbb{E}^s_{\nu,t} \left [ \int_{t+\tau}^T \| \varphi^2_r \|_{L^2(\Lambda)}^2 \mathrm{d}r \right ]^{\frac12}\\
			&\qquad\qquad + \mathbb{E}^s_{\nu,t} \left [ \int_{t+\tau}^T \| q_r \|_{L_2(\Xi,L^2(\Lambda))}^2 \mathrm{d}r \right ]^{\frac12} \mathbb{E}^s_{\nu,t} \left [ \int_{t+\tau}^T \| \psi^2_r \|_{L_2(\Xi,L^2(\Lambda))}^2 \mathrm{d}r \right ]^{\frac12}
		\end{split}
	\end{equation}
	Since the first factor is finite in each case, the claim follows from the remainder estimates \eqref{order2}.
\end{proof}

\begin{lemma}\label{duality22}
	Let $y^{\tau,z}$ be the process given by equation \eqref{zerothorder1}, and let $P^{\eta}$ be the mollified second order adjoint state. Then it holds for almost every $t\in [s,T]$
	\begin{equation}
		\begin{split}
			&\mathbb{E}^s_{\nu,t} \left [ \int_{t+\tau}^T \int_{\Lambda} (l_{xx}(\bar x_r(\lambda), \bar u_r) + b_{xx}(\bar x_r(\lambda), \bar u_r) p_r(\lambda)) y^{\tau,z}_r(\lambda) y^{\tau,z}_r(\lambda) \mathrm{d}\lambda \mathrm{d}r \right ]\\
			&\quad + \mathbb{E}^s_{\nu,t} \left [ \int_{t+\tau}^T \int_{\Lambda} \langle \sigma_{xx}(\bar x_r(\lambda), \bar u_r),q_r \rangle_{L_2(\Xi,\mathbb{R})} y^{\tau,z}_r(\lambda) y^{\tau,z}_r(\lambda) \mathrm{d}\lambda \mathrm{d}r \right ]\\
			&\quad + \mathbb{E}^s_{\nu,t} \left [ \int_{\Lambda^2} h^{\eta}_{xx}(\lambda,\mu) y^{\tau,z}_T(\lambda) y^{\tau,z}_T(\mu) \mathrm{d}\lambda \mathrm{d}\mu \right ]\\
			&= \mathbb{E}^s_{\nu,t} \left [ \int_{\Lambda^2} P^{\eta}_{t+\tau}(\lambda,\mu) y^{\tau,z}_{t+\tau}(\lambda) y^{\tau,z}_{t+\tau}(\mu) \mathrm{d}\lambda \mathrm{d}\mu \right ] \\
			&\quad + \mathbb{E}^s_{\nu,t} \left [ \int_{t+\tau}^T \langle P^{\eta}_r, \Phi^{\tau,z}_r \rangle_{L^2(\Lambda^2)} + \langle Q^{\eta}_r, \Psi^{\tau,z}_r \rangle_{L_2(\Xi,L^2(\Lambda^2))} \mathrm{d}r \right ],
		\end{split}
	\end{equation}
	$\mathbb{P}$--almost surely, where
	\begin{equation}
		(\Phi^{\tau,z},\Psi^{\tau,z}) \in L^2([s,T]\times\Omega; L^2(\Lambda^2)) \times L^2([s,T]\times \Omega; L_2(\Xi,L^2(\Lambda^2)))
	\end{equation}
	are given by
	\begin{equation}
		\begin{split}
			\Phi^{\tau,z}_r(\lambda,\mu) :=& y^{\tau,z}_r(\lambda) \varphi^{1}_r(\mu) + y^{\tau,z}_r(\mu) \varphi^{1}_r(\lambda)\\
			&+ \sigma_x(\bar x_r(\lambda), \bar u_r) y^{\tau,z}_r(\lambda) \psi^{1}_r(\mu) + \sigma_x(\bar x_r(\mu), \bar u_r) y^{\tau,z}_r(\mu) \psi^{1}_r(\lambda)\\
			&+ \psi^{1}_r(\lambda) \psi^{1}_r(\mu),
		\end{split}
	\end{equation}
	and
	\begin{equation}
		\Psi^{\tau,z}_r(\lambda,\mu) := y^{\tau,z}_r(\lambda) \psi^{1}_r(\mu) + y^{\tau,z}_r(\mu) \psi^{1}_r(\lambda).
	\end{equation}
	Furthermore, we have for almost every $t\in [s,T]$
	\begin{equation}\label{remainderP}
		\mathbb{E}^s_{\nu,t} \left [ \int_{t+\tau}^T \langle P^{\eta}_r , \Phi^{\tau,z}_r \rangle_{L^2(\Lambda^2)} + \langle Q^{\eta}_r , \Psi^{\tau,z}_r \rangle_{L_2(\Xi,L^2(\Lambda^2))} \mathrm{d}r \right ] = o\left ( \tau + \| z \|_{L^2(\Lambda)}^2 \right )
	\end{equation}
	$\mathbb{P}$--almost surely.
\end{lemma}

\begin{proof}
	In order to invoke the second order adjoint state, we use the same idea as in the proof of Peng's maximum principle for SPDEs, see \cite{stannat2020}. We rewrite the quadratic terms in $y^{\tau,z}$ in the following way
	\begin{equation}\label{delta}
		\begin{split}
			\langle p_r , b_{xx}(\bar x_r, \bar u_r) y^{\tau,z}_r y^{\tau,z}_r \rangle_{L^2(\Lambda)} &= \int_{\Lambda} p_r(\lambda) b_{xx}(\bar x_r(\lambda), \bar u_r) y^{\tau,z}_r(\lambda) y^{\tau,z}_r(\lambda)\mathrm{d}\lambda \\
			&= \int_{\Lambda} p_r(\lambda) b_{xx}(\bar x_r(\lambda), \bar u_r) \delta(Y^{\tau,z}_r)(\lambda) \mathrm{d}\lambda,
		\end{split}
	\end{equation}
	where $Y^{\tau,z}_r(\lambda,\mu):= y^{\tau,z}_r(\lambda) y^{\tau,z}_r(\mu)$ and $\delta$ is the operator introduced in Section \ref{adjointequations}.
	
	Next, let us derive the equation for $Y^{\tau,z}$. Similar to the calculation in \cite{stannat2020}, we have
	\begin{equation}\label{secondvariationalequation}
		\begin{cases}
			\mathrm{d}Y^{\tau,z}_r(\lambda,\mu) = [\Delta Y^{\tau,z}_r(\lambda,\mu) + ( b_{x}(\bar x_r(\lambda),\bar u_r) + b_{x}(\bar x_r(\mu),\bar u_r) ) Y^{\tau,z}_r(\lambda,\mu)\\
			\qquad\qquad\qquad + \langle \sigma_{x}(\bar x_r(\lambda),\bar u_r), \sigma_{x}(\bar x_r(\mu),\bar u_r) \rangle_{L_2(\Xi,\mathbb{R})} Y^{\tau,z}_r(\lambda,\mu) + \Phi^{\tau,z}_r(\lambda,\mu) ]\mathrm{d}r\\
			\qquad\qquad\qquad + [ ( \sigma_{x}(\bar x_r(\lambda),\bar u_r) + \sigma_{x}(\bar x_r(\mu),\bar u_r) ) Y^{\tau,z}_r(\lambda,\mu) + \Psi^{\tau,z}_r(\lambda,\mu) ] \mathrm{d}W_r\\
			Y^{\tau,z}_{t+\tau} =(z+\bar x_t -\bar x_{t+\tau}) \otimes (z+\bar x_t-\bar x_{t+\tau}).
		\end{cases}
	\end{equation}
	We want to apply It\^o's formula to the product $\langle Y^{\tau,z}_r,P_r \rangle_{L^2(\Lambda^2)}$. Since $P$ itself is not sufficiently regular, we take the mollified second order adjoint process given by equation \eqref{mollifiedsecondorderadjoint}, instead. Applying It\^o's formula to $\langle Y^{\tau,z}_r,P^{\eta}_r \rangle_{L^2(\Lambda^2)}$ yields the duality relation.
	
	For the first term in the remainder estimate, we observe
	\begin{equation}
		\mathbb{E}^s_{\nu,t} \left [ \int_{t+\tau}^T \langle P^{\eta}_r , \Phi^{\tau,z}_r \rangle_{L^2(\Lambda^2)} \mathrm{d}r \right ]\leq \mathbb{E}^s_{\nu,t} \left [ \int_{t+\tau}^T \| P^{\eta}_r \|_{L^2(\Lambda^2)}^2 \mathrm{d}r \right ]^{\frac12} \mathbb{E}^s_{\nu,t} \left [ \int_{t+\tau}^T \| \Phi^{\tau,z}_r \|_{L^2(\Lambda^2)}^2 \mathrm{d}r \right ]^{\frac12}.
	\end{equation}
	Since the first factor is finite, the claim follows from the a priori estimates in Lemma \ref{aprioriy} and \eqref{order1}.
	
	For the second term in the remainder estimate, we have
	\begin{equation}\label{qestimate}
		\begin{split}
			& \mathbb{E}^s_{\nu,t} \left [ \int_{t+\tau}^T \langle Q^{\eta}_r , \Psi^{\tau,z}_r \rangle_{L_2(\Xi,L^2(\Lambda^2))} \mathrm{d}r \right ]\\
			&\leq \mathbb{E}^s_{\nu,t} \left [ \int_{t+\tau}^T \| Q^{\eta}_r \|_{L_2(\Xi,L^2(\Lambda^2))}^2 \mathrm{d}r \right ]^{\frac12} \mathbb{E}^s_{\nu,t} \left [ \int_{t+\tau}^T \| \Psi^{\tau,z}_r \|_{L_2(\Xi,L^2(\Lambda^2))}^2 \mathrm{d}r \right ]^{\frac12}
		\end{split}
	\end{equation}
	Since
	\begin{equation}
		\begin{split}
			\| \Psi^{\tau,z}_r \|^2_{L_2(\Xi,L^2(\Lambda^2))} &= \| y^{\tau,z}_r \otimes \psi^1_r + \psi^1_r \otimes y^{\tau,z}_r \|^2_{L_2(\Xi,L^2(\Lambda^2))} \\
			&= \| y^{\tau,z}_r \|_{L^2(\Lambda)}^2 \| \psi^1_r \|_{L_2(\Xi,L^2(\Lambda))}^2 + \| y^{\tau,z}_r \|_{L^2(\Lambda)}^2 \| \psi^1_r \|_{L_2(\Xi,L^2(\Lambda))}^2,
		\end{split}
	\end{equation}
	the claim follows again from Lemma \ref{aprioriy} and \eqref{order1}.
\end{proof}

\subsubsection{Time-Increments}\label{timeincrements}

The next two lemmata address the time increment.

\begin{lemma}\label{firsttimeincrement}
	It holds for almost every $t\in [s,T]$
	\begin{equation}
		\begin{split}
			&\mathbb{E}^s_{\nu,t} \left [ \left \langle p_{t+\tau}, \bar x_{t+\tau}-\bar x_t \right \rangle_{L^2(\Lambda)} \right ]\\
			&= \tau \mathbb{E}^s_{\nu,t} \left [ \left \langle p_t, \Delta\bar x_t + b(\bar x_t, \bar u_t) \right \rangle_{H^1_0(\Lambda)\times H^{-1}(\Lambda)} + \left \langle q_t ,\sigma(\bar x_t, \bar u_t) \right \rangle_{L_2(\Xi,L^2(\Lambda))} \right ]+ o(\tau),
		\end{split}
	\end{equation}
	$\tau>0$, $\mathbb{P}$--almost surely.
\end{lemma}

\begin{proof}
	Applying It\^o's formula for variational solutions to SPDEs and taking the conditional expectation, we obtain
	\begin{equation}\label{ptimeincrement}
		\begin{split}
			& \mathbb{E}^s_{\nu,t} \left [ \left \langle p_{t+\tau}, \bar x_{t+\tau}-\bar x_t \right \rangle_{L^2(\Lambda)} \right ]\\
			&= \mathbb{E}^s_{\nu,t} \left [ \int_t^{t+\tau} \left \langle \Delta p_r, \bar x_t \right \rangle_{H^{-1}(\Lambda) \times H^1_0(\Lambda)} + \left \langle p_r, b(\bar x_r, \bar u_r) \right \rangle_{L^2(\Lambda)} + \langle \sigma(\bar x_r, \bar u_r), q_r \rangle_{L_2(\Xi,L^2(\Lambda))} \mathrm{d}r \right ]\\
			&\quad - \mathbb{E}^s_{\nu,t} \left [ \int_t^{t+\tau} \left \langle (\bar x_r - \bar x_t), b_x(\bar x_r,\bar u_r) p_r + \langle \sigma_x(\bar x_r,\bar u_r), q_r \rangle + l_x(\bar x_r,\bar u_x) \right \rangle_{L^2(\Lambda)} \mathrm{d}r \right ].
		\end{split}
	\end{equation}
	Note that the stochastic integrals vanish under the expectation. For the first term, we have
	\begin{equation}
		\begin{split}
			&\left | \mathbb{E}^s_{\nu,t} \left [ \frac{1}{\tau} \int_t^{t+\tau} \langle \Delta ( p_r - p_t ), \bar x_t \rangle_{H^{-1}(\Lambda) \times H^1_0(\Lambda)} \mathrm{d}r \right ] \right |\\
			&\leq \mathbb{E}^s_{\nu,t} \left [ \frac{1}{\tau} \int_t^{t+\tau} \| p_r-p_t \|_{H^1_0(\Lambda)} \| \bar x_t \|_{H^1_0(\Lambda)} \mathrm{d}r \right ]\\
			&\leq \mathbb{E}^s_{\nu,t} \left [ \frac{1}{\tau} \int_t^{t+\tau} \| p_r-p_t \|_{H^1_0(\Lambda)}^2 \mathrm{d}r \right ]^{\frac12} \| \bar x_t \|_{H^1_0(\Lambda)}.
		\end{split}
	\end{equation}
	By Lebesgue's differentiation theorem, we have for almost every $t\in [s,T]$
	\begin{equation}
		\frac{1}{\tau} \int_t^{t+\tau} \mathbb{E}^s_{\nu,t} \left [ \| p_r-p_t \|_{H^1_0(\Lambda)}^2 \right ] \mathrm{d}r \to 0
	\end{equation}
	$\mathbb{P}$--almost surely. Hence, we obtain for almost every $t\in [s,T]$
	\begin{equation}
		\mathbb{E}^s_{\nu,t} \left [ \frac{1}{\tau} \int_t^{t+\tau} \langle p_r, \bar x_t \rangle_{H^{-1}(\Lambda) \times H^1_0(\Lambda)} \mathrm{d}r \right ] \to \langle p_t ,\Delta \bar x_t \rangle_{H^1_0(\Lambda)\times H^{-1}(\Lambda)}
	\end{equation}
	$\mathbb{P}$--almost surely along some subsequence. Arguing similarly for the second and third term of equation \eqref{ptimeincrement} and noticing that the last line is of order $o(\tau)$ concludes the proof.
\end{proof}

\begin{lemma}\label{secondtimeincrement}
	It holds for almost every $t\in [s,T]$
	\begin{equation}
		\begin{split}
			&\mathbb{E}^s_{\nu,t} \left [ \int_{\Lambda^2} P^{\eta}_{t+\tau}(\lambda,\mu) (\bar x_{t+\tau}-\bar x_t)(\lambda) (\bar x_{t+\tau}-\bar x_t)(\mu) \mathrm{d}\lambda \mathrm{d}\mu \right ]\\
			&= \tau \mathbb{E}^s_{\nu,t} \left [ \int_{\Lambda^2} P^{\eta}_t(\lambda,\mu) \langle \sigma(\bar x_t(\lambda),\bar u_t), \sigma(\bar x_t(\mu),\bar u_t)\rangle_{L_2(\Xi,\mathbb{R})} \mathrm{d}\lambda \mathrm{d}\mu \right ] + o(\tau),
		\end{split}
	\end{equation}
	$\tau>0$, $\mathbb{P}$--almost surely.
\end{lemma}

\begin{proof}
	The equation for the tensor product $ (\bar x_{t+\tau}-\bar x_t) \otimes (\bar x_{t+\tau}-\bar x_t)$ is
	\begin{equation}
		\begin{split}
			&\mathrm{d}((\bar x_{t+\tau}-\bar x_t)(\lambda) (\bar x_{t+\tau}-\bar x_t) (\mu) )\\
			&= \left [ (\bar x_{t+\tau} - \bar x_t)(\lambda) \Delta\bar x_{t+\tau}(\mu) + (\bar x_{t+\tau} - \bar x_t)(\mu) \Delta\bar x_{t+\tau}(\lambda) \right ] \mathrm{d}\tau\\
			&\quad+ \left [ (\bar x_{t+\tau}-\bar x_t)(\lambda) b(\bar x_{t+\tau}(\mu),\bar u_{t+\tau}) + (\bar x_{t+\tau}-\bar x_t)(\mu) b(\bar x_{t+\tau}(\lambda),\bar u_{t+\tau}) \right ] \mathrm{d}\tau\\
			&\quad+ \langle \sigma(\bar x_{t+\tau}(\lambda),\bar u_{t+\tau}), \sigma(\bar x_{t+\tau}(\mu),\bar u_{t+\tau}) \rangle_{L_2(\Xi,\mathbb{R})} \mathrm{d}\tau\\
			&\quad + \left [ (\bar x_{t+\tau}-\bar x_t)(\lambda) \sigma(\bar x_{t+\tau}(\mu),\bar u_{t+\tau}) + (\bar x_{t+\tau}-\bar x_t)(\mu) \sigma(\bar x_{t+\tau}(\lambda),\bar u_{t+\tau}) \right ] \mathrm{d}W_{\tau}.
		\end{split}
	\end{equation}
	Again, applying It\^o's formula for variational solutions to SPDEs and taking the conditional expectation yields
	\begin{equation}
		\begin{split}
			&\mathbb{E}^s_{\nu,t} \left [ \left \langle P^{\eta}_{t+\tau}, (\bar x_{t+\tau} -\bar x_t)\otimes (\bar x_{t+\tau}-\bar x_t) \right \rangle \right ]\\
			&= \mathbb{E}^s_{\nu,t} \left [ \int_t^{t+\tau} \langle P^{\eta}_r, \mathrm{d} ((\bar x_r -\bar x_t)\otimes (\bar x_r-\bar x_t)) \rangle + \int_t^{t+\tau} \langle (\bar x_r -\bar x_t)\otimes (\bar x_r-\bar x_t), \mathrm{d}P^{\eta}_r \rangle \right ]\\
			&\quad + \mathbb{E}^s_{\nu,t} \left [ \langle P^{\eta}_{\cdot}, (\bar x_{\cdot} -\bar x_t)\otimes (\bar x_{\cdot} -\bar x_t) \rangle_{t+\tau} \right ],
		\end{split}
	\end{equation}
	where
	\begin{equation}
		\begin{split}	
			&\mathbb{E}^s_{\nu,t} \left [ \int_t^{t+\tau} \langle P^{\eta}_r, \mathrm{d} ((\bar x_r -\bar x_t)\otimes (\bar x_r-\bar x_t)) \rangle \right ]\\
			&= \mathbb{E}^s_{\nu,t} \left [ \int_t^{t+\tau} \langle P^{\eta}_r, (\bar x_r - \bar x_t)(\lambda) \Delta\bar x_r(\mu) + (\bar x_r - \bar x_t)(\mu) \Delta\bar x_r(\lambda) \rangle \mathrm{d}r \right ]\\
			&\quad + \mathbb{E}^s_{\nu,t} \left [ \int_t^{t+\tau} \langle P^{\eta}_r, (\bar x_r -\bar x_t)(\lambda) b(\bar x_r(\mu),\bar u_r) + (\bar x_r -\bar x_t)(\mu) b(\bar x_r(\lambda),\bar u_r) \rangle \mathrm{d}r \right ]\\
			&\quad + \mathbb{E}^s_{\nu,t} \left [ \int_t^{t+\tau} \langle P^{\eta}_r , \langle \sigma(\bar x_r(\lambda),\bar u_r), \sigma(\bar x_r(\mu),\bar u_r) \rangle_{L_2(\Xi,\mathbb{R})} \rangle \mathrm{d}r \right ],
		\end{split}
	\end{equation}
	and
	\begin{equation}
		\begin{split}
			&\mathbb{E}^s_{\nu,t}\! \left [ \int_t^{t+\tau} \!\!\! \langle (\bar x_r\! -\!\bar x_t)\otimes (\bar x_r\!-\!\bar x_t), \mathrm{d}P^{\eta}_r \rangle \right ]\\
			&= \mathbb{E}^s_{\nu,t} \!\left [ \int_t^{t+\tau}\!\!\! \langle (\bar x_r\!-\!\bar x_t)(\lambda) (\bar x_r\!-\!\bar x_t) (\mu) , \Delta P^{\eta}_r(\lambda,\mu) \rangle \mathrm{d}r \right ]\\
			&\quad + \mathbb{E}^s_{\nu,t} \!\left [ \int_t^{t+\tau}\!\!\! \langle (\bar x_r\!-\!\bar x_t)(\lambda) (\bar x_r\!-\!\bar x_t) (\mu) , (b_x(\bar x_r(\lambda),\bar u_r) + b_x(\bar x_r(\mu),\bar u_r)) P^{\eta}_r(\lambda,\mu) \rangle \mathrm{d}r \right ]\\
			&\quad+ \mathbb{E}^s_{\nu,t} \!\left [ \int_t^{t+\tau}\!\!\! \langle (\bar x_r\!-\!\bar x_t)(\lambda) (\bar x_r\!-\!\bar x_t) (\mu) , \langle \sigma_{x}(\bar x_r(\lambda), \bar u_r), \sigma_{x}(\bar x_r(\mu), \bar u_r) \rangle_{L_2(\Xi,\mathbb{R})} P^{\eta}_r(\lambda,\mu) \rangle \mathrm{d}r \right ]\\
			&\quad+ \mathbb{E}^s_{\nu,t}\! \left [ \int_t^{t+\tau}\!\!\! \langle (\bar x_r\!-\!\bar x_t)(\lambda) (\bar x_r\!-\!\bar x_t) (\mu) , \langle \sigma_{x} (\bar x_r(\lambda),\bar u_r)\! +\! \sigma_{x} (\bar x_r(\mu),\bar u_r), Q^{\eta}_r(\lambda,\mu)\rangle \rangle \mathrm{d}r \right ]\\
			&\quad+ \mathbb{E}^s_{\nu,t}\! \left [ \int_t^{t+\tau}\!\!\! \langle (\bar x_r\!-\!\bar x_t)(\lambda) (\bar x_r\!-\!\bar x_t) (\mu) , \delta^{\ast} ( l_{xx}(\bar x_r(\lambda), \bar u_r) + b_{xx}(\bar x_r(\lambda), \bar u_r) p_r(\lambda)) \rangle \mathrm{d}r \right ]\\
			&\quad+ \mathbb{E}^s_{\nu,t}\! \left [ \int_t^{t+\tau} \!\!\! \langle (\bar x_r\!-\!\bar x_t)(\lambda) (\bar x_r\!-\!\bar x_t) (\mu) , \delta^{\ast} ( \langle \sigma_{xx}(\bar x_r(\lambda),\bar u_r), q_r \rangle_{L_2(\Xi,\mathbb{R})}) \rangle \mathrm{d}r \right ],
		\end{split}
	\end{equation}
	and
	\begin{equation}
		\begin{split}
			&\mathbb{E}^s_{\nu,t} \left [ \langle P^{\eta}_{\cdot}, (\bar x_{\cdot} -\bar x_t)\otimes (\bar x_{\cdot} -\bar x_t) \rangle_{t+\tau} \right ]\\
			&= \mathbb{E}^s_{\nu,t} \left [ \int_t^{t+\tau} \langle (\bar x_r-\bar x_t)\otimes \sigma(\bar x_r,\bar u_r) + \sigma(\bar x_r,\bar u_r) \otimes (\bar x_r-\bar x_t) , Q^{\eta}_r \rangle_{L_2(\Xi,L^2(\Lambda^2))} \mathrm{d}r \right ].
		\end{split}
	\end{equation}
	Except
	\begin{equation}
		\mathbb{E}^s_{\nu,t} \left [ \int_t^{t+\tau} \langle P^{\eta}_r , \langle \sigma(\bar x_r(\lambda),\bar u_r), \sigma(\bar x_r(\mu),\bar u_r) \rangle_{L_2(\Xi,\mathbb{R})} \rangle \mathrm{d}r \right ],
	\end{equation}
	the integrand in each summand contains the term $\bar x_r-\bar x_t$. Therefore, arguing as in the proof of Lemma \ref{firsttimeincrement}, when dividing by $\tau$ and taking the limit $\tau\downarrow 0$, this is the only remaining term.
\end{proof}

\subsubsection{Mixed Time- and Space-Increments}\label{mixedtimeandspaceincrements}

\begin{lemma}\label{mixedtimeincrement}
	It holds for almost every $t\in [s,T]$
	\begin{equation}
		\mathbb{E}^s_{\nu,t} \left [ \int_{\Lambda^2} P^{\eta}_{t+\tau}(\lambda,\mu) z (\lambda) ( \bar x_{t+\tau} - \bar x_t ) (\mu) \mathrm{d}\lambda \mathrm{d}\mu \right ] = o\left ( \tau + \|z\|_{L^2(\Lambda)}^2 \right ),
	\end{equation}
	$\tau>0$, $z\in L^2(\Lambda)$, $\mathbb{P}$--almost surely.
\end{lemma}

\begin{proof}
	First note
	\begin{equation}\label{separate}
		\begin{split}
			&\mathbb{E}^s_{\nu,t} \left [ \int_{\Lambda^2} P^{\eta}_{t+\tau}(\lambda,\mu) z(\lambda) ( \bar x_{t+\tau} - \bar x_t ) (\mu) \mathrm{d}\lambda \mathrm{d}\mu \right ]\\
			&= \mathbb{E}^s_{\nu,t} \left [ \int_{\Lambda^2} (P^{\eta}_{t+\tau}-P^{\eta}_t)(\lambda,\mu) z(\lambda) ( \bar x_{t+\tau} - \bar x_t ) (\mu) \mathrm{d}\lambda \mathrm{d}\mu \right ]\\
			&\quad + \mathbb{E}^s_{\nu,t} \left [ \int_{\Lambda^2} P^{\eta}_t(\lambda,\mu) z(\lambda) ( \bar x_{t+\tau} - \bar x_t ) (\mu) \mathrm{d}\lambda \mathrm{d}\mu \right ].
		\end{split}
	\end{equation}
	For the first expectation, we have
	\begin{equation}
		\begin{split}
			&\mathbb{E}^s_{\nu,t} \left [ \int_{\Lambda^2} (P^{\eta}_{t+\tau}-P^{\eta}_t)(\lambda,\mu) z(\lambda) ( \bar x_{t+\tau} - \bar x_t ) (\mu) \mathrm{d}\lambda \mathrm{d}\mu \right ]\\
			&\leq \mathbb{E}^s_{\nu,t} \left [ \| P^{\eta}_{t+\tau}-P^{\eta}_t \|_{L^2(\Lambda^2)}^2 \right ]^{\frac12}  \| z \|_{L^2(\Lambda)} \mathbb{E}^s_{\nu,t} \left [ \| \bar x_{t+\tau} - \bar x_t \|^2_{L^2(\Lambda)} \right ]^{\frac12}.
		\end{split}
	\end{equation}
	Since the second and third term are each of order $O\left ( \sqrt{\tau + \|z\|_{L^2(\Lambda)}^2} \right )$, and $P^{\eta}$ is continuous with values in $L^2(\Lambda^2)$ $\mathbb{P}$-almost surely, the whole expression is of order $o\left ( \tau + \|z\|_{L^2(\Lambda)}^2 \right )$.
	
	For the second expectation in equation \eqref{separate}, we have
	\begin{equation}
		\begin{split}
			&\mathbb{E}^s_{\nu,t} \left [ \int_{\Lambda^2} P^{\eta}_t(\lambda,\mu) z(\lambda) ( \bar x_{t+\tau} - \bar x_t ) (\mu) \mathrm{d}\lambda \mathrm{d}\mu \right ]\\
			&\leq \| P^{\eta}_t \|_{L^2(\Lambda^2)} \| z \|_{L^2(\Lambda)} \left \| \mathbb{E}^s_{\nu,t} \left [ \bar x_{t+\tau} - \bar x_t \right ] \right \|_{L^2(\Lambda)}.
		\end{split}
	\end{equation}
	Since
	\begin{equation}
		\left \| \mathbb{E}^s_{\nu,t} \left [ \bar x_{t+\tau} - \bar x_t \right ] \right \|_{L^2(\Lambda)}^2 = 2 \int_t^{t+\tau} \langle \mathbb{E}^s_{\nu,t} [ \bar x_r - \bar x_t ] , \mathbb{E}^s_{\nu,t} [ \Delta \bar x_r + b(\bar x_r,\bar u_r) ] \rangle_{H^1_0(\Lambda)\times H^{-1}(\Lambda)} \mathrm{d}r,
	\end{equation}
	by Lebesgue's differentiation theorem, we have for almost every $t\in [s,T]$
	\begin{equation}
		\left \| \mathbb{E}^s_{\nu,t} \left [ \bar x_{t+\tau} - \bar x_t \right ] \right \|_{L^2(\Lambda)} = o\left (\sqrt{\tau} \right )
	\end{equation}
	$\mathbb{P}$-almost surely.	Therefore, for almost every $t\in [s,T]$
	\begin{equation}
		\mathbb{E}^s_{\nu,t} \left [ \int_{\Lambda^2} P^{\eta}_t(\lambda,\mu) z(\lambda) ( \bar x_{t+\tau} - \bar x_t ) (\mu) \mathrm{d}\lambda \mathrm{d}\mu \right ] = o\left ( \tau + \|z\|_{L^2(\Lambda)}^2 \right ),
	\end{equation}
	$\mathbb{P}$--almost surely, which concludes the proof.
\end{proof}

\begin{lemma}\label{timeincrementpmixed}
	It holds for almost every $t\in [s,T]$
	\begin{equation}
		\mathbb{E}^s_{\nu,t} \left [ \int_{\Lambda} ( p_{t+\tau}(\lambda) - p_t(\lambda) ) z(\lambda) \mathrm{d}\lambda \right ] = o\left ( \tau + \|z\|_{L^2(\Lambda)}^2 \right ),
	\end{equation}
	$\tau>0$, $z\in L^2(\Lambda)$, $\mathbb{P}$--almost surely.
\end{lemma}

\begin{proof}
	As in the proof of Lemma \ref{mixedtimeincrement}, we have
	\begin{equation}
		\left \| \mathbb{E}^s_{\nu,t} \left [ p_{t+\tau}- p_t \right ] \right \|_{L^2(\Lambda)} = o\left (\sqrt{\tau} \right ).
	\end{equation}
	Therefore,
	\begin{equation}
		\begin{split}
			&\mathbb{E}^s_{\nu,t} \left [ \int_{\Lambda} ( p_{t+\tau}(\lambda) - p_t(\lambda) ) z(\lambda) \mathrm{d}\lambda \right ]\\
			&\leq \left \| \mathbb{E}^s_{\nu,t} \left [ p_{t+\tau}(\lambda) - p_t(\lambda) \right ] \right \|_{L^2(\Lambda)} \| z\|_{L^2(\Lambda)} \\
			&= o\left ( \tau + \|z\|_{L^2(\Lambda)}^2 \right ),
		\end{split}
	\end{equation}
	which concludes the proof.
\end{proof}

\subsubsection{Proof of Theorem \ref{spacetime}}\label{proof}

Fix $t\in [s,T]$ such that all the preceding lemmata hold $\mathbb{P}$--almost surely, and let $\tau > 0$. From the Markov property of the solution of the state equation, it follows that
\begin{equation}
	V(t,\bar x_t) = \mathbb{E}^s_{\nu,t} \left [  \int_t^T \int_{\Lambda} l(\bar x_r(\lambda),\bar u_r) \mathrm{d}\lambda \mathrm{d}r + \int_{\Lambda} h(\bar x_T(\lambda)) \mathrm{d}\lambda \right ],
\end{equation}
see \cite[Chapter 4, Lemma 3.2 and Theorem 3.4]{yong1999} and \cite[Section 2.3.3]{fabbri2017} for the necessary modifications in the infinite dimensional case. Furthermore, we have
\begin{equation}
	V(t+\tau, \bar x_t+z) \leq \mathbb{E}^s_{\nu,t} \left [ \int_{t+\tau}^T \int_{\Lambda} l( x^{\tau,z}_r(\lambda), \bar u_r) \mathrm{d}\lambda \mathrm{d}r + \int_{\Lambda} h(x^{\tau,z}_T(\lambda)) \mathrm{d}\lambda \right ]
\end{equation}
Thus, we obtain for almost every $t\in [s,T]$\begin{equation}
	\begin{split}
		&V(t+\tau, \bar x_t+z) - V(t,\bar x_t)\\
		&\leq \mathbb{E}^s_{\nu,t} \left [ - \int_t^{t+\tau} \int_{\Lambda} l(\bar x_r(\lambda), \bar u_r) \mathrm{d}\lambda \mathrm{d}r + \int_{t+\tau}^T \int_{\Lambda} l( x^{\tau,z}_r(\lambda), \bar u_r) - l(\bar x_r(\lambda),\bar u_r) \mathrm{d}\lambda \mathrm{d}r \right ]\\
		&\quad + \mathbb{E}^s_{\nu,t} \left [ \int_{\Lambda} h(x^{\tau,z}_T(\lambda)) - h(\bar x_T(\lambda)) \mathrm{d}\lambda \right ]\\
		&= \mathbb{E}^s_{\nu,t} \left [ - \int_t^{t+\tau} \int_{\Lambda} l(\bar x_r(\lambda), \bar u_r) \mathrm{d}\lambda \mathrm{d}r + \int_{t+\tau}^T \int_{\Lambda} l_x(\bar x_r(\lambda), \bar u_r) y^{\tau,z}_r(\lambda) \mathrm{d}\lambda \mathrm{d}r \right ]\\
		&\quad+ \mathbb{E}^s_{\nu,t} \left [ \int_{\Lambda} h_x(\bar x_T(\lambda)) y^{\tau,z}_T(\lambda) \mathrm{d}\lambda + \frac12 \int_{t+\tau}^T \int_{\Lambda} l_{xx}(\bar x_r(\lambda), \bar u_r) y^{\tau,z}_r(\lambda) y^{\tau,z}_r(\lambda) \mathrm{d}\lambda \mathrm{d}r \right ]\\
		&\quad+ \mathbb{E}^s_{\nu,t} \left [ \frac12 \int_{\Lambda} h_{xx}(\bar x_T(\lambda)) y^{\tau,z}_T(\lambda) y^{\tau,z}_T(\lambda) \mathrm{d}\lambda \right ] +  o\left (\tau + \|z\|_{L^2(\Lambda)}^2 \right ),
	\end{split}
\end{equation}
$\mathbb{P}$-almost surely, where the remainder terms of the Taylor expansion are of order $o\left (\tau + \|z\|_{L^2(\Lambda)}^2 \right )$ for the same reason as in \eqref{order1}. Using the duality relations from Lemma \ref{duality11} and Lemma \ref{duality22}, and the estimates for the remainder terms of the duality relations, we obtain
\begin{equation}\label{test}
	\begin{split}
		&V(t+\tau, \bar x_t+z) - V(t,\bar x_t)\\
		&\leq \mathbb{E}^s_{\nu,t} \left [ - \int_t^{t+\tau} \int_{\Lambda} l(\bar x_r(\lambda), \bar u_r) \mathrm{d}\lambda \mathrm{d}r + \int_{\Lambda} p_{t+\tau}(\lambda) y^{\tau,z}_{t+\tau}(\lambda) \mathrm{d}\lambda \right ] \\
		&\quad + \mathbb{E}^s_{\nu,t} \left [ \frac12 \int_{\Lambda^2} P^{\eta}_{t+\tau}(\lambda,\mu) y^{\tau,z}_{t+\tau}(\lambda) y^{\tau,z}_{t+\tau}(\mu) \mathrm{d}\lambda \mathrm{d}\mu \right ] \\
		&\quad + \frac12 \mathbb{E}^s_{\nu,t} \left [ \int_{\Lambda} h_{xx}(\bar x_T(\lambda)) y^{\tau,z}_T(\lambda) y^{\tau,z}_T(\lambda) \mathrm{d}\lambda - \int_{\Lambda^2} h^{\eta}_{xx}(\lambda,\mu) y^{\tau,z}_T(\lambda) y^{\tau,z}_T(\mu) \mathrm{d}\lambda \mathrm{d}\mu \right ] \\
		&\quad + o\left (\tau + \|z\|_{L^2(\Lambda)}^2 \right ).
	\end{split}
\end{equation}
Plugging in the initial condition
\begin{equation}
	\begin{split}
		y^{\tau,z}_{t+\tau} &= z+\bar x_t -\bar x_{t+\tau}\\
		&= z - \left ( \int_t^{t+\tau} \Delta \bar x_r + b(\bar x_r,\bar u_r) \mathrm{d}r + \int_t^{t+\tau} \sigma(\bar x_r,\bar u_r) \mathrm{d}W_r \right ),
	\end{split}
\end{equation}
we have
\begin{equation}
	\begin{split}
		&V(t+\tau,\bar x_t+ z) - V(t,\bar x_t)\\
		&\leq \mathbb{E}^s_{\nu,t} \left [ - \int_t^{t+\tau} \int_{\Lambda} l(\bar x_r(\lambda), \bar u_r) \mathrm{d}\lambda \mathrm{d}r + \int_{\Lambda} p_{t+\tau}(\lambda) z(\lambda) \mathrm{d}\lambda \right ]\\
		&\quad -\mathbb{E}^s_{\nu,t} \left [ \int_{\Lambda} p_{t+\tau}(\lambda) ( \bar x_{t+\tau} - \bar x_t ) (\lambda) \mathrm{d}\lambda \right ]\\
		&\quad+ \frac12 \mathbb{E}^s_{\nu,t} \left [ \int_{\Lambda^2} P^{\eta}_{t+\tau}(\lambda,\mu) z(\lambda) z(\mu) \mathrm{d}\lambda \mathrm{d}\mu \right ] \\
		&\quad+ \frac12 \mathbb{E}^s_{\nu,t} \left [ \int_{\Lambda^2} P^{\eta}_{t+\tau}(\lambda,\mu) ( \bar x_{t+\tau} - \bar x_t ) (\lambda) ( \bar x_{t+\tau} - \bar x_t ) (\mu) \mathrm{d}\lambda \mathrm{d}\mu \right ] \\
		&\quad- \mathbb{E}^s_{\nu,t} \left [ \int_{\Lambda^2} P^{\eta}_{t+\tau}(\lambda,\mu) z(\lambda) ( \bar x_{t+\tau} - \bar x_t ) (\mu) \mathrm{d}\lambda \mathrm{d}\mu \right ]\\
		&\quad+ \frac12 \mathbb{E}^s_{\nu,t} \left [ \int_{\Lambda} h_{xx}(\bar x_T(\lambda)) y^{\tau,z}_T(\lambda) y^{\tau,z}_T(\lambda) \mathrm{d}\lambda - \int_{\Lambda^2} h^{\eta}_{xx}(\lambda,\mu) y^{\tau,z}_T(\lambda) y^{\tau,z}_T(\mu) \mathrm{d}\lambda \mathrm{d}\mu \right ]\\
		&\quad+ o\left (\tau + \|z\|_{L^2(\Lambda)}^2 \right )
	\end{split}
\end{equation}
Now we apply Lemma \ref{firsttimeincrement}, Lemma \ref{secondtimeincrement} and Lemma \ref{mixedtimeincrement}, and obtain
\begin{equation}
	\begin{split}
		&V(t+\tau, \bar x_t+z) - V(t, \bar x_t) \\
		&\leq \tau \mathbb{E}^s_{\nu,t} \left [ \vphantom{\int_{\Lambda^2}} - \left \langle p_t, \Delta\bar x_t + b(\bar x_t,\bar u_t) \right \rangle_{H^1_0(\Lambda)\times H^{-1}(\Lambda)} - \langle q_t, \sigma(\bar x_t,\bar u_t) \rangle_{L_2(\Xi,L^2(\Lambda))}\right.\\
		&\quad- \int_{\Lambda} l(\bar x_t(\lambda),\bar u_t) \mathrm{d}\lambda + \left.\frac12 \int_{\Lambda^2} P^{\eta}_t(\lambda,\mu) \langle \sigma(\bar x_t(\lambda),\bar u_t), \sigma(\bar x_t(\mu),\bar u_t) \rangle_{L_2(\Xi,\mathbb{R})} \mathrm{d}\lambda \mathrm{d}\mu \right ]\\
		&\quad+ \mathbb{E}^s_{\nu,t} \left [ \int_{\Lambda} p_{t+\tau}(\lambda) z(\lambda) \mathrm{d}\lambda + \frac12 \int_{\Lambda^2} P^{\eta}_{t+\tau}(\lambda,\mu) z(\lambda) z(\mu) \mathrm{d}\lambda \mathrm{d}\mu \right ]\\
		&\quad + \frac12 \mathbb{E}^s_{\nu,t} \left [ \int_{\Lambda} h_{xx}(\bar x_T(\lambda)) y^{\tau,z}_T(\lambda) y^{\tau,z}_T(\lambda) \mathrm{d}\lambda - \int_{\Lambda^2} h^{\eta}_{xx}(\lambda,\mu) y^{\tau,z}_T(\lambda) y^{\tau,z}_T(\mu) \mathrm{d}\lambda \mathrm{d}\mu \right ]\\
		&\quad + o\left ( \tau + \|z\|_{L^2(\Lambda)}^2 \right ).
	\end{split}
\end{equation}
Using the definition of the Hamiltonian, we can rewrite this as
\begin{equation}
	\begin{split}
		&V(t+\tau, \bar x_t+z) - V(t, \bar x_t)\\
		&\leq \tau \left [ -\langle \Delta \bar x_t ,p_t \rangle_{H^{-1}(\Lambda)\times H^1_0(\Lambda)} -\mathcal{H}(\bar x_t,\bar u_t,p_t,P^{\eta}_t) - \text{tr} \left ( \sigma(\bar x_t,\bar u_t)^{\ast} \left [ q_t - P^{\eta}_t \sigma(\bar x_t,\bar u_t) \right ] \right ) \right ]\\
		&\quad+ \mathbb{E}^s_{\nu,t} \left [ \int_{\Lambda} p_{t+\tau}(\lambda) z(\lambda) \mathrm{d}\lambda + \frac12 \int_{\Lambda^2} P^{\eta}_{t+\tau}(\lambda,\mu) z(\lambda) z(\mu) \mathrm{d}\lambda \mathrm{d}\mu \right ]\\
		&\quad + \frac12 \mathbb{E}^s_{\nu,t} \left [ \int_{\Lambda} h_{xx}(\bar x_T(\lambda)) y^{\tau,z}_T(\lambda) y^{\tau,z}_T(\lambda) \mathrm{d}\lambda - \int_{\Lambda^2} h^{\eta}_{xx}(\lambda,\mu) y^{\tau,z}_T(\lambda) y^{\tau,z}_T(\mu) \mathrm{d}\lambda \mathrm{d}\mu \right ]\\
		&\quad + o\left ( \tau + \|z\|_{L^2(\Lambda)}^2 \right ).
	\end{split}
\end{equation}
Adding a zero and rearranging terms yields
\begin{equation}
	\begin{split}
		&V(t+\tau, \bar x_t+z) - V(t, \bar x_t) - \left ( - \langle \Delta \bar x_t ,p_t \rangle_{H^{-1}(\Lambda)\times H^1_0(\Lambda)} - \mathcal{G}(t,\bar x_t,\bar u_t) \right ) \tau\\
		&- \int_{\Lambda} p_t(\lambda) z(\lambda) \mathrm{d}\lambda - \frac12 \int_{\Lambda^2} P_t(\lambda,\mu) z(\lambda) z(\mu) \mathrm{d}\lambda \mathrm{d}\mu\\
		&\leq \frac12 \text{tr} \left ( \sigma(\bar x_t,\bar u_t)^{\ast}  ( P^{\eta}_t - P_t) \sigma(\bar x_t,\bar u_t) \right ) \tau\\
		&\quad+ \mathbb{E}^s_{\nu,t} \left [ \int_{\Lambda} ( p_{t+\tau}(\lambda) - p_t(\lambda) ) z(\lambda) \mathrm{d}\lambda \right ]\\
		&\quad+ \mathbb{E}^s_{\nu,t} \left [ \frac12 \int_{\Lambda^2} ( P^{\eta}_{t+\tau}(\lambda,\mu) - P_t(\lambda,\mu) ) z(\lambda) z(\mu) \mathrm{d}\lambda \mathrm{d}\mu \right ]\\
		&\quad + \frac12 \mathbb{E}^s_{\nu,t} \left [ \int_{\Lambda} h_{xx}(\bar x_T(\lambda)) y^{\tau,z}_T(\lambda) y^{\tau,z}_T(\lambda) \mathrm{d}\lambda - \int_{\Lambda^2} h^{\eta}_{xx}(\lambda,\mu) y^{\tau,z}_T(\lambda) y^{\tau,z}_T(\mu) \mathrm{d}\lambda \mathrm{d}\mu \right ]\\
		&\quad + o\left ( \tau + \|z\|_{L^2(\Lambda)}^2 \right ).
	\end{split}
\end{equation}
Using Lemma \ref{timeincrementpmixed} and elementary estimates for the right-hand side, we obtain
\begin{equation}
	\begin{split}
		&V(t+\tau, \bar x_t+z) - V(t, \bar x_t) - \left ( - \langle \Delta \bar x_t ,p_t \rangle_{H^{-1}(\Lambda)\times H^1_0(\Lambda)} - \mathcal{G}(t,\bar x_t,\bar u_t) \right ) \tau\\
		&- \int_{\Lambda} p_t(\lambda) z(\lambda) \mathrm{d}\lambda - \frac12 \int_{\Lambda^2} P_t(\lambda,\mu) z(\lambda) z(\mu) \mathrm{d}\lambda \mathrm{d}\mu\\
		&\leq \frac12 \| P^{\eta}_t - P_t \|_{L^2(\Lambda^2)} \| \sigma(\bar x_t,\bar u_t) \|_{L_2(\Xi,L^2(\Lambda))}^2 \tau + \frac12 \|P^{\eta}_{t+\tau}-P_t\|_{L^2(\Lambda^2)} \|z\|_{L^2(\Lambda)}^2\\
		&\quad + \frac12 \mathbb{E}^s_{\nu,t} \left [ \int_{\Lambda} h_{xx}(\bar x_T(\lambda)) y^{\tau,z}_T(\lambda) y^{\tau,z}_T(\lambda) \mathrm{d}\lambda - \int_{\Lambda^2} h^{\eta}_{xx}(\lambda,\mu) y^{\tau,z}_T(\lambda) y^{\tau,z}_T(\mu) \mathrm{d}\lambda \mathrm{d}\mu \right ]\\
		&\quad + o\left ( \tau + \|z\|_{L^2(\Lambda)}^2 \right ).
	\end{split}
\end{equation}
Now, let $(\tau_k,z_k) \to 0$, $\tau_k >0$, be a sequence, which realizes the limit superior of the left-hand side divided by $\tau_k+ \|z_k\|^2_{L^2(\Lambda)}$. By Lemma \ref{regularityterminalcondition} and using the compact embedding $H^{\gamma}_0(\Lambda) \subset\subset L^2(\Lambda)$ (see, e.g., \cite[Theorem 4.54]{demengel2012}), we can extract a subsequence --again denoted by $(\tau_k,z_k)$-- such that $y^{\tau_k,z_k}_T/\sqrt{\tau_k+\|z_k\|^2_{L^2(\Lambda)}}$ converges in $L^2(\Lambda)$ to some limit $\tilde y_T$. Therefore, dividing the inequality by $\tau_k + \|z_k\|^2_{L^2(\Lambda)}$ and sending $(\tau_k,z_k)$ to zero yields
\begin{equation}
	\begin{split}
		&\limsup_{\tau\downarrow 0, z\to 0} \frac{1}{\tau +\|z\|^2_{L^2(\Lambda)}} \Bigg \{ V(t+\tau, \bar x_t+z) - V(t, \bar x_t) - \left ( - \langle \Delta \bar x_t ,p_t \rangle_{L^2(\Lambda)} - \mathcal{G}(t,\bar x_t,\bar u_t) \right ) \tau\\
		&\quad - \int_{\Lambda} p_t(\lambda) z(\lambda) \mathrm{d}\lambda + \frac12 \int_{\Lambda^2} P_t(\lambda,\mu) z(\lambda) z(\mu) \mathrm{d}\lambda \mathrm{d}\mu \Bigg \} \\
		&\leq \frac12 \| P^{\eta}_t - P_t \|_{L^2(\Lambda^2)} \| \sigma(\bar x_t,\bar u_t) \|_{L_2(\Xi,L^2(\Lambda))}^2 + \frac12 \|P^{\eta}_t-P_t\|_{L^2(\Lambda^2)}\\
		&\quad + \frac12 \mathbb{E}^s_{\nu,t} \left [ \int_{\Lambda} h_{xx}(\bar x_T(\lambda)) \tilde y_T(\lambda) \tilde y_T(\lambda) \mathrm{d}\lambda - \int_{\Lambda^2} h^{\eta}_{xx}(\lambda,\mu) \tilde y_T(\lambda) \tilde y_T(\mu) \mathrm{d}\lambda \mathrm{d}\mu \right ].
	\end{split}
\end{equation}
Taking the limit $\eta \to 0$, the right-hand side vanishes, which concludes the proof of the first claim.

The second claim \eqref{4.85} follows along the same lines as in the finite dimensional case with similar modifications as above. \qed

\subsection{Space-Derivatives}\label{sectionspace}

In this section, we consider the case with differentials only in the spatial variable. To this end, we first recall the definition of viscosity super- and subdifferentials.

\begin{definition}
	For $v\in C([s,T]\times L^2(\Lambda))$ the first order viscosity superdifferential in the space-variable of $v$ at $(t,x)\in [s,T]\times L^2(\Lambda)$ is the set
	\begin{equation}
		D^{1,+}_{x} v(t,x) := \left \{ p \in L^2(\Lambda) \middle |	\limsup_{z\to 0} \frac{v(t,x+z) - v(t,x) - \langle p, z\rangle_{L^2(\Lambda)}}{ \|z\|_{L^2(\Lambda)}} \leq 0 \right  \}.
	\end{equation}
	The first order viscosity subdifferential $D^{1,-}_{x} v$ is defined analogously with the $\limsup$ replaced by $\liminf$ and the $\leq$ replaced by $\geq$.
\end{definition}

Concerning the first order derivative, we obtain the following result.

\begin{corollary}\label{firstorderspace}
	It holds for almost every $t\in [s,T]$
	\begin{equation}
		D^{1,-}_x V(t,\bar x_t) \subset \{ p_t \} \subset D^{1,+}_x V(t,\bar x_t)
	\end{equation}
	$\mathbb{P}$--almost surely.
\end{corollary}
This follows from Theorem \ref{spacetime} by restricting the $\limsup$ to $\tau = 0$ and estimating
\begin{equation}
	\langle z, P_t z \rangle_{L^2(\Lambda)} \leq \|P_t \|_{L^2(\Lambda^2)} \|z \|_{L^2(\Lambda)}^2.
\end{equation}
Next, we consider the second order viscosity differentials in the space-variable.
\begin{definition}
	For $v\in C([s,T]\times L^2(\Lambda))$ the second order viscosity superdifferential in the space-variable of $v$ at $(t,x)\in [s,T]\times L^2(\Lambda)$ is the set
	\begin{multline}
		D^{2,+}_{x} v(t,x) := \left \{ (p,P) \in L^2(\Lambda) \times \mathcal{S}(L^2(\Lambda)) \middle | \vphantom{\frac{v(\hat t,x) - v(\hat t, \hat x) - \langle p, x-\hat x\rangle_{L^2(\Lambda)} - \frac12 \langle x-\hat x, P(x-\hat x) \rangle_{L^2(\Lambda)}}{ \| x-\hat x\|_{L^2(\Lambda)}^2}}\right.\\
		\left. \limsup_{z\to 0} \frac{v(t,x+z) - v(t,x) - \langle p, z \rangle_{L^2(\Lambda)} - \frac12 \langle z, Pz \rangle_{L^2(\Lambda)}}{ \| z\|_{L^2(\Lambda)}^2} \leq 0 \right  \}.
	\end{multline}
	The second order viscosity subdifferential $D^{2,-}_{x} v$ is defined analogously with the $\limsup$ replaced by $\liminf$ and the $\leq$ replaced by $\geq$.
\end{definition}

\begin{corollary}\label{space}
	It holds for almost every $t\in [s,T]$
	\begin{equation}\label{claim1}
		\{ p_t \} \times \mathcal{S}_{\succeq P_t}(L^2(\Lambda)) \subset D_x^{2,+} V(t,\bar x_t),
	\end{equation}
	$\mathbb{P}$--almost surely, and for almost every $t\in [s,T]$
	\begin{equation}\label{claim2}
		D_x^{2,-} V(t,\bar x_t) \subset \{ p_t \} \times \mathcal{S}_{\preceq P_t}(L^2(\Lambda)).
	\end{equation}
	$\mathbb{P}$--almost surely, where
	\begin{equation}
		\mathcal{S}_{\succeq P_t}(L^2(\Lambda)) := \{ S\in \mathcal{S}(L^2(\Lambda)): S-P_t \text{ is a positive operator}\},
	\end{equation}
	and $\mathcal{S}_{\preceq P_t}(L^2(\Lambda))$ is defined analogously.
\end{corollary}
The proof follows again from Theorem \ref{spacetime} by restricting the $\limsup$ to $\tau =0$.

\begin{remark}
	This result together with our main result Theorem \ref{spacetime} extends the necessary condition in Pontryagin's maximum principle by adjoint state inclusions. In the deterministic infinite dimensional case, the differential inclusion for the first order adjoint state was proven in \cite{cannarsa1992}. In \cite{zhou19912}, Zhou proved the corresponding result for the case of controlled SPDEs. However, going only to first order does not fully reflect the stochastic nature of the problem. In the finite dimensional case, Zhou proved the differential inclusion for the second order adjoint state in \cite{zhou1991}. The main obstacle in the generalization of this result to infinite dimensions was the missing representation of the second order adjoint state, which appears due to the It\^o correction term arising in stochastic calculus.
\end{remark}

\subsection{Time-Derivatives}\label{sectiontime}

In this section, we consider the case with differentials only in the time-variable.
\begin{definition}
	For $v\in C([s,T]\times L^2(\Lambda))$ the first order viscosity superdifferential in the time-variable of $v$ at $(t,x)\in [s,T)\times L^2(\Lambda)$ is the set
	\begin{equation}
		D^{1,+}_{t+} v(t,x) := \left \{ G \in \mathbb{R} \middle |	\limsup_{\tau\downarrow 0} \frac{v(t+\tau,x) - v(t,x) - G\tau}{\tau} \leq 0 \right  \}.
	\end{equation}
	The first order viscosity subdifferential $D^{1,-}_{t+} v$ is defined analogously with the $\limsup$ replaced by $\liminf$ and the $\leq$ replaced by $\geq$.
\end{definition}

\begin{corollary}\label{time}
	It holds for almost every $t\in [s,T]$
	\begin{equation}\label{4.67}
		- \langle \Delta\bar x_t, p_t\rangle_{H^{-1}(\Lambda)\times H^1_0(\Lambda)} - \mathcal{G}(t,\bar x_t, \bar u_t) \in D_{t+}^{1,+} V(t,\bar x_t)
	\end{equation}
	$\mathbb{P}$--almost surely.
\end{corollary}
The proof follows from Theorem \ref{spacetime} by restricting the $\limsup$ to $z=0$.

\subsection{Non-Positivity of the Correction Term}\label{nonpositivity}

As another corollary of Theorem \ref{spacetime}, we derive non-positivity of the correction term arising in non-smooth stochastic control problems.
\begin{corollary}\label{gleqh}
	Let
	\begin{equation}\label{higherregularity}
		\begin{cases}
			G\in L^2([s,T]\times\Omega;\mathbb{R})\\
			p\in L^2([s,T]\times\Omega;H^1_0(\Lambda))\\
			P\in L^2([s,T]\times\Omega;L_2(L^2(\Lambda)))
		\end{cases}
	\end{equation}
	be adapted processes such that for almost every $t\in [s,T]$
	\begin{equation}
		(G_t,p_t,P_t) \in D^{1,2,+}_{t+,x} V(t, \bar x_t)
	\end{equation}
	$\mathbb{P}$-almost surely.	Then it holds for almost every $t\in [s,T]$
	\begin{equation}\label{necessary}
		G_t + \langle \Delta \bar x_t, p_t\rangle_{H^{-1}(\Lambda)\times H^1_0(\Lambda)} + \mathcal{H}( \bar x_t, \bar u_t,p_t,P_t) \geq 0,
	\end{equation}
	$\mathbb{P}$--almost surely.
\end{corollary}

\begin{remark}
	The higher regularity assumptions $p$ and $P$ in \eqref{higherregularity} are necessary due to the unbounded term in \eqref{necessary}. Notice that the adjoint states given by \eqref{firstadjoint} and \eqref{secondadjoint}, respectively, satisfy this higher regularity. In case $p$ and $P$ are the adjoint states, and
	\begin{equation}
		G_t = - \langle \Delta \bar x_t, p_t\rangle_{H^{-1}(\Lambda)\times H^1_0(\Lambda)} - \mathcal{G}(t,\bar x_t, \bar u_t),
	\end{equation}
	equation \eqref{necessary} is equivalent to
	\begin{equation}
		\mathcal{G}(t,\bar x_t, \bar u_t) \leq \mathcal{H}(t, \bar x_t, \bar u_t, p_t, P_t).
	\end{equation}
\end{remark}

The proof of Corollary \ref{gleqh} in the finite dimensional case uses the fact that $V$ is a viscosity solution of the HJB equation combined with the following correspondence between test functions and points in the parabolic viscosity superdifferential, see e.g. \cite[Chapter V, Lemma 4.1]{fleming2006}.

\begin{proposition}\label{testfunction}
	Let $v:(s,T)\times L^2(\Lambda)\to \mathbb{R}$ be upper semicontinuous. For $(t,x)\in (s,T)\times L^2(\Lambda)$, it holds $(G,p,P) \in D_{t+,x}^{1,2,+} v(t,x)$ if and only if there exists a function $\phi \in C^{1,2}((s,T)\times L^2(\Lambda))$ such that $v-\phi$ attains its strict global maximum over the set $[t,T)\times L^2(\Lambda)$ at the point $(t,x)$, and
	\begin{equation}\label{testfunctionproperty}
		(\phi(t,x), \partial_t\phi(t,x), D\phi (t,x), D^2 \phi(t,x)) = (v(t,x),G,p,P).
	\end{equation}
\end{proposition}

\begin{remark}\label{remarktestfunction}
	The proof of Proposition \ref{testfunction} in the infinite dimensional case follows along the same lines as in the finite dimensional case, see \cite[Proposition 2.1]{chen2022}. However, in the theory of $B$-continuous viscosity solutions for equations involving unbounded operators, test functions are of the form $\psi=\varphi +h(t,\|x\|_{L^2(\Lambda)})$ with additional assumptions imposed on $\varphi$ and $h$. In particular, $\Delta D\varphi$ is required to be continuous, see Definition \ref{testfunctiondefinition} below. If $p$ is not in the domain of the Laplace operator, equation \eqref{testfunctionproperty} cannot be satisfied for such a test function. Therefore, in the proof of Corollary \ref{gleqh}, we cannot apply Proposition \ref{testfunction} to obtain a test function in the sense of Definition \ref{testfunctiondefinition}. Instead, we have to carry out the argument from the finite dimensional case by hand. In addition to dealing with technical difficulties already arising in the proof of the verification theorem within the framework of viscosity solutions (see \cite{gozzi2005}), we have to perform a delicate regularity analysis due to the unbounded operator in the state equation \eqref{stateequation}.
\end{remark}

Now let us get to the proof of Corollary \ref{gleqh}.
\begin{proof}
	Fix $t \in[s,T]$ such that
	\begin{equation}
		(G_t,p_t,P_t) \in D^{1,2,+}_{t+,x} V(t,\bar x_t)
	\end{equation} 
	$\mathbb{P}$--almost surely. Following the idea from the finite dimensional case (see \cite[Chapter V, Lemma 4.1]{fleming2006}), we define for $\beta>0$
	\begin{multline}
		g(\beta) := \sup \Bigg \{ \frac{\left (V(t+\tau,\bar x_t+z) - V(t,\bar x_t) - G_t \tau - \langle p_t,z \rangle_{L^2(\Lambda)} - \frac12 \langle z, P_t z \rangle_{L^2(\Lambda)} \right )^+}{\left (\tau^2+\|z\|_{L^2(\Lambda)}^4 \right )^{\frac12}} \Bigg |\\
		(t+\tau,z)\in (s,T)\times L^2(\Lambda), 0<(\tau^2+\|z\|_{L^2(\Lambda)}^4)^{\frac12} \leq \beta \Bigg \},
	\end{multline}
	and set $g(0):=0$. Since
	\begin{equation}
		\limsup_{\tau\downarrow 0,z\to 0} \frac{\tau+\|z\|_{L^2(\Lambda)}^2}{\left (\tau^2 + \|z\|_{L^2(\Lambda)}^4 \right )^{\frac12}} <\infty
	\end{equation}
	and $(G_t,p_t,P_t)\in D^{1,2,+}_{t+,x} V(t,\bar x_t)$, $g$ is continuous and non-decreasing on $[0,\infty)$. Using $g$, we define
	\begin{equation}\label{aandf}
		\begin{split}
			a(\theta,x) &:= \left ( (\theta-t)^2 +\|x-\bar x_t\|_{L^2(\Lambda)}^4 \right )^{\frac12}\\
			F(a) &:= \frac{2}{3a} \int_a^{2a} \int_{\xi}^{2\xi} g(\beta) \mathrm{d}\beta \mathrm{d}\xi, \quad F(0)=0
		\end{split}
	\end{equation}
	and construct the test function $\phi : (s,T)\times L^2(\Lambda) \to \mathbb{R}$ as
	\begin{multline}
		\phi(\theta,x) := F(a(\theta,x)) + V(t,\bar x_t) + G_t(\theta-t)\\
		+ \langle p_t ,x-\bar x_t \rangle_{L^2(\Lambda)} + \frac12 \langle x-\bar x_t, P_t(x-\bar x_t) \rangle_{L^2(\Lambda)}.
	\end{multline}
	In order to obtain higher regularity for the term $D\phi (r,\bar x_r)$, we replace the process $P$ by an approximation.	Let $(e_l)_{l\geq 1} \subset H^1_0(\Lambda)$ be an orthonormal basis of $L^2(\Lambda)$ and define $P^n x := \sum_{l=1}^n \langle Px, e_l \rangle_{L^2(\Lambda)} e_l$. Then we have for every $n\in \mathbb{N}$
	\begin{itemize}
		\item $P^n \in L^2([0,T]\times \Omega ; L(L^2(\Lambda)))$, and $\|P^n_t \|_{L_2(L^2(\Lambda))} \leq \|P_t \|_{L_2(L^2(\Lambda))}$ $\mathrm{d}t\otimes\mathbb{P}$--almost everywhere;
		\item $P^n_t(H^1_0(\Lambda)) \subset H^1_0(\Lambda)$ $\mathrm{d}t\otimes\mathbb{P}$--almost everywhere;
		\item $P^n_t \to P_t$ in the uniform operator topology $\mathrm{d}t\otimes\mathbb{P}$--almost everywhere.
	\end{itemize}
	Note that these conditions also imply $P^n \to P$ in $L^2([0,T]\times\Omega; L(L^2(\Lambda)))$. We approximate $\phi$ by
	\begin{multline}\label{phin}
		\phi^n(\theta,x) := F(a(\theta,x)) + V(t, x^{\ast}_t) + G_t(\theta-t)\\
		+ \langle p_t ,x- x^{\ast}_t \rangle_{L^2(\Lambda)} + \frac12 \langle x- x^{\ast}_t, P^n_t(x- x^{\ast}_t) \rangle_{L^2(\Lambda)}.
	\end{multline}
	
	Since $V-\phi$ attains its maximum at $(t,\bar x_t)$ and by the dynamic programming principle, we have for every $\tau \geq 0$,
	\begin{equation}\label{viscosity}
		\begin{split}
			0 &\geq \mathbb{E}^s_{\nu,t} \left [ V(t+\tau,\bar x_{t+\tau}) - \phi(t+\tau,\bar x_{t+\tau}) - (V(t,\bar x_t) - \phi(t,\bar x_t)) \right ]\\
			&= \mathbb{E}^s_{\nu,t} \left [ - \int_t^{t+\tau} \int_{\Lambda} l(\bar x_r(\lambda),\bar u_r) \mathrm{d}\lambda \mathrm{d}r - \phi^{n}(t+\tau,\bar x_{t+\tau}) + \phi^{n}(t,\bar x_t) \right ]\\
			&\quad + \mathbb{E}^s_{\nu,t} \left [ \phi^{n}(t+\tau,\bar x_{t+\tau}) - \phi(t+\tau,\bar x_{t+\tau}) - \phi^{n}(t,\bar x_t) + \phi(t,\bar x_t) \right ].
		\end{split}
	\end{equation}
	Now, we want to apply It\^o's formula to $\phi^n$. However, $\phi^n$ implicitly depends on $\omega$ via $\bar x_t$, $G_t$, $p_t$ and $P_t$. Therefore, we fix an $\omega\in\Omega$ and switch to the probability space $(\Omega,\mathcal{F}^{\nu}, \mathbb{P}(\,\cdot\,|\mathcal{F}^s_{\nu,t})(\omega))$, where $\mathbb{P}(\,\cdot\,|\mathcal{F}^s_{\nu,t})(\cdot)$ denotes the regular conditional probability given $\mathcal{F}^s_{\nu,t}$. On this probability space, $\bar x_t$, $G_t$, $p_t$ and $P_t$ are almost surely constant, and are equal to $\bar x_t(\omega)$, $G_t(\omega)$, $p_t(\omega)$ and $P_t(\omega)$. See also \cite{gozzi2005} for more details on this. In the following, we denote by $\mathbb{E}^s_{\nu,t} [\, \cdot\, ](\omega)$ the expectation with respect to $\mathbb{P}(\,\cdot\,|\mathcal{F}^s_{\nu,t})(\omega)$. Thus, we derive
	\begin{equation}\label{itosformula}
		\begin{split}
			& \mathbb{E}^s_{\nu,t} \left [ \phi^{n}(t+\tau,\bar x_{t+\tau}) - \phi^{n}(t,\bar x_t) \right ](\omega)\\
			&= \mathbb{E}^s_{\nu,t} \left [ \int_t^{t+\tau} \partial_{\theta} \phi^{n}(r,\bar x_r) + \langle D \phi^{n}(r,\bar x_r), \Delta \bar x_r + b(\bar x_r,\bar u_r) \rangle_{H^1_0(\Lambda)\times H^{-1}(\Lambda)} \right.\\
			&\qquad\qquad\qquad\qquad\qquad\qquad + \left. \frac12 \vphantom{\int_t^{t+\tau}} \text{tr}\left (\sigma(\bar x_r,\bar u_r)^{\ast} D^2 \phi^{n}(r,\bar x_r) \sigma(\bar x_r,\bar u_r) \right ) \mathrm{d}r \right ](\omega).
		\end{split}
	\end{equation}
	From Lemma \ref{asconvergence} it follows
	\begin{equation}
		\begin{split}
			&\mathbb{E}^s_{\nu,t} \left [ \frac{1}{\tau_k} \int_t^{t+\tau_k} \partial_{\theta} \phi^{n}(r,\bar x_r) + \langle D \phi^{n}(r,\bar x_r), \Delta \bar x_r + b(\bar x_r,\bar u_r) \rangle_{H^1_0(\Lambda)\times H^{-1}(\Lambda)} \right.\\
			&\qquad\qquad\qquad\qquad\qquad\qquad + \left. \frac12 \text{tr}\left (\sigma(\bar x_r,\bar u_r)^{\ast} D^2 \phi^{n}(r,\bar x_r) \sigma(\bar x_r,\bar u_r) \right ) \mathrm{d}r \right ](\omega) \\
			&\to G_t(\omega) + \langle p_t(\omega), \Delta \bar x_t(\omega) + b( \bar x_t(\omega), \bar u_t(\omega)) \rangle\\
			&\qquad\qquad\qquad\qquad\qquad\qquad + \frac12 \text{tr}(\sigma ( \bar x_t(\omega), \bar u_t(\omega))^{\ast} P^{n}_t(\omega) \sigma( \bar x_t(\omega), \bar u_t(\omega))),
		\end{split}
	\end{equation}
	$\mathrm{d}t\otimes\mathbb{P}$--almost everywhere.
	
	For the last line in equation \eqref{viscosity}, we note that
	\begin{equation}
		\phi^{n}(\theta,x) - \phi(\theta,x) = \frac12 \langle x- \bar x_t, (P^{n}_t - P_t)(x-\bar x_t) \rangle_{L^2(\Lambda)}.
	\end{equation}
	Therefore, $\mathbb{E}^s_{\nu,t} [ \phi^{n}(t,\bar x_t) - \phi(t,\bar x_t) ]$ vanishes, and
	\begin{equation}\label{Pvanishes}
		\left | \mathbb{E}^s_{\nu,t} \left [ \phi^{n}(t+\tau_k,\bar x_{t+\tau_k}) - \phi(t+\tau_k,\bar x_{t+\tau_k}) \right ] \right | \leq \frac12 \| P^{n}_t - P_t \|_{L^2(\Lambda^2)}^2 \mathbb{E}^s_{\nu,t} \left [ \| \bar x_{t+\tau_k} - \bar x_t \|_{L^2(\Lambda)}^2 \right ].
	\end{equation}
	Dividing by $\tau_k$, taking the limit $k\to\infty$, and exploiting equation \eqref{timeregularityx}, we obtain
	\begin{equation}
		\lim_{k\to\infty} \frac{1}{\tau_k} \left | \mathbb{E}^s_{\nu,t} \left [ \phi^{n}(t+\tau_k,\bar x_{t+\tau_k}) - \phi(t+\tau_k,\bar x_{t+\tau_k}) \right ] \right | \leq C \| P^{n}_t - P_t \|_{L^2(\Lambda^2)}^2.
	\end{equation}
	Altogether, we derive from equation \eqref{viscosity} for almost every $t\in [s,T]$
	\begin{equation}
		\begin{split}
			0 \geq& \lim_{k\to\infty} \frac{1}{\tau_k} \mathbb{E}^s_{\nu,t} \left [ V(t+\tau_k,\bar x_{t+\tau_k}) - \phi(t+\tau_k,\bar x_{t+\tau_k}) - V(t,\bar x_t) + \phi(t,\bar x_t) \right ]\\
			\geq& \text{tr} \left ( \sigma(\bar x_t,\bar u_t)^{\ast} \left [ q_t - P^{n}_t \sigma(\bar x_t,\bar u_t) \right ] \right ) - C \| P^{n}_t - P_t \|_{L^2(\Lambda^2)}^2
		\end{split}
	\end{equation}
	$\mathbb{P}$--almost surely. Taking the limit $n \to \infty$ concludes the proof.
\end{proof}

\begin{lemma}\label{asconvergence}
	For almost every $t\in[s,T]$, it holds
	\begin{equation}\label{convergence}
		\begin{split}
			&\mathbb{E}^s_{\nu,t} \left [ \frac{1}{\tau} \int_t^{t+\tau} \partial_{\theta} \phi^{n}(r,\bar x_r) + \langle D \phi^{n}(r,\bar x_r), \Delta \bar x_r + b(\bar x_r,\bar u_r) \rangle_{H^1_0(\Lambda)\times H^{-1}(\Lambda)} \right.\\
			&\qquad\qquad\qquad\qquad\qquad\qquad + \left. \frac12 \text{tr}\left (\sigma(\bar x_r,\bar u_r)^{\ast} D^2 \phi^{n}(r,\bar x_r) \sigma(\bar x_r,\bar u_r) \right ) \mathrm{d}r \right ](\omega) \\
			&\to G_t(\omega) + \langle p_t(\omega), \Delta \bar x_t(\omega) + b( \bar x_t(\omega), \bar u_t(\omega)) \rangle\\
			&\qquad\qquad\qquad\qquad\qquad\qquad+ \frac12 \text{tr}(\sigma ( \bar x_t(\omega), \bar u_t(\omega))^{\ast} P^{n}_t(\omega) \sigma( \bar x_t(\omega), \bar u_t(\omega))),
		\end{split}
	\end{equation}
	as $\tau \to 0$, $\mathbb{P}$--almost surely. Here, $\phi^{n}$ is given by \eqref{phin}.
\end{lemma}

\begin{proof}
	We fix $t\in [s,T]$ such that \eqref{4.84} holds and the convergence in all the applications of Lebesgue's differentiation theorem below holds for $t$. Note that the set of such $t$ is a set of full measure. Now, let us first discuss some properties of $\phi^{n}$ that can be proven similarly as in the finite dimensional case, see \cite[Chapter V, Lemma 4.1]{fleming2006}. Recall the definition of the approximated test function
	\begin{multline}
		\phi^{n}(\theta,x) := F(a(\theta,x)) + V(t,\bar x_t) + G_t(\theta-t)\\
		+ \langle p_t ,x-\bar x_t \rangle_{L^2(\Lambda)} + \frac12 \langle x-\bar x_t, P^{n}_t(x-\bar x_t) \rangle_{L^2(\Lambda)}.
	\end{multline}
	where $F$ and $a$ are given by \eqref{aandf}. The derivatives of $F$ and $a$ are given by
	\begin{equation}
		\begin{cases}
			F^{\prime}(a) = \frac{4}{3a} \int_{2a}^{4a} g(\xi) \mathrm{d}\xi - \frac{2}{3a} \int_a^{2a} g(\xi)\mathrm{d}\xi - \frac{1}{a} F(a),\quad F^{\prime}(0)=0\\
			F^{\prime\prime}(a) = \frac{2}{3a}(8g(4a) - 6g(2a) + g(a)) - \frac{2}{a} F^{\prime}(a),\quad F^{\prime\prime}(0)=0
		\end{cases}
	\end{equation}
	and
	\begin{equation}
		\begin{cases}
			\partial_{\theta} a(\theta,x) = \frac{\theta - t}{a(\theta,x)}\\
			Da(\theta,x) = \frac{2 \|x - \bar x_t\|_{L^2(\Lambda)}^2}{a(\theta,x)} (x - \bar x_t)\\
			D^2a(\theta,x) = \left ( \frac{4}{a(\theta,x)} - \frac{4\| x- \bar x_t \|_{L^2(\Lambda)}^4}{a(\theta,x)^3} \right ) (x-\bar x_t)\otimes (x-\bar x_t) + \frac{2 \| x-\bar x_t \|_{L^2(\Lambda)}^2}{a(\theta,x)} \langle \cdot,\cdot \rangle_{L^2(\Lambda)}.
		\end{cases}	
	\end{equation}
	The first and second derivative of $\phi^{n}$ are given by
	\begin{equation}
		\begin{cases}
			\partial_{\theta} \phi^{n}(\theta,x) = \frac{\theta-t}{a(\theta,x)} F^{\prime}(a(\theta,x)) + G_t\\
			D \phi^{n}(\theta,x) = F^{\prime}(a(\theta,x)) Da(\theta,x) + p_t + P^{n}_t(x - \bar x_t)\\
			D^2\phi^{n}(\theta,x) = F^{\prime\prime}(a(\theta,x)) Da(\theta,x) \otimes Da(\theta,x) + F^{\prime}(a(\theta,x)) D^2a(\theta,x) + P^{n}_t
		\end{cases}
	\end{equation}
	Thus, $\phi^{n} \in C^{1,2}((s,T)\times L^2(\Lambda))$, and
	\begin{equation}\label{phi}
		(\phi^{n}(t,\bar x_t), \partial_{\theta}\phi^{n}(t,\bar x_t), D\phi^{n} (t,\bar x_t), D^2 \phi^{n}(t,\bar x_t)) = (V(t,\bar x_t),G_t,p_t,P^{n}_t).
	\end{equation}
	Furthermore, $F(a(\theta,x)) \in C^{1,2}([s,T]\times L^2(\Lambda))$ with vanishing time derivative and first and second order space derivative at $(\theta,x) = (t,\bar x_t)$. Furthermore, note that
	\begin{equation}\label{fprime}
		|F^{\prime}(a(\theta,x))| \leq C\left (1+ \| x \|_{L^2(\Lambda)}^2\right )
	\end{equation}
	and
	\begin{equation}\label{fprimeprime}
		| F^{\prime\prime}(a(\theta,x)) | \leq C\left (1+\|x\|_{L^2(\Lambda)}^2\right ).
	\end{equation}
	
	Now let us start with the proof of \eqref{convergence}. Let $(\tau_k)_{k\in\mathbb{N}}$ be an arbitrary sequence converging to zero. We fix $\omega\in\Omega$, and, as described in the discussion before \eqref{itosformula}, we switch to the probability space $(\Omega,\mathcal{F}^{\nu}, \mathbb{P}(\,\cdot\,|\mathcal{F}^s_{\nu,t})(\omega))$. First, we consider the term involving the time derivative of $\phi^{n}$. We have
	\begin{equation}
		\partial_{\theta} \phi^{n}(r,\bar x_r) = \frac{r-t}{a(r,\bar x_r)} F^{\prime}(a(r,\bar x_r)) + G_t.
	\end{equation}
	By the almost sure continuity of
	\begin{equation}
		r \mapsto F^{\prime}(a(r,\bar x_r)) \frac{r-t}{a(r,\bar x_r)},
	\end{equation}
	and the dominated convergence theorem, we have
	\begin{equation}\label{timephi}
		\lim_{k\to\infty} \mathbb{E}^s_{\nu,t} \left [ \left | \frac{1}{\tau_k} \int_t^{t+\tau_k} F^{\prime}(a(r,\bar x_r)) \frac{r-t}{a(r,\bar x_r)} \mathrm{d}r \right | \right ](\omega) =0.
	\end{equation}
	
	Now, let us consider the first space derivative of $\phi^{n}$. We have
	\begin{equation}
		D \phi^{n}(r,\bar x_r) = F^{\prime}(a(r,\bar x_r)) Da(r,\bar x_r) + p_t + P^{n}_t(\bar x_r - \bar x_t),
	\end{equation}
	thus we have to show
	\begin{equation}\label{spacederivativefirst}
		\lim_{k\to\infty} \mathbb{E}^s_{\nu,t} \left [ \left | \frac{1}{\tau_k} \int_t^{t+\tau_k} \langle F^{\prime}(a(r,\bar x_r)) Da(r,\bar x_r), \Delta \bar x_r + b(\bar x_r,\bar u_r) \rangle_{H^1_0(\Lambda)\times H^{-1}(\Lambda)} \mathrm{d}r \right | \right ](\omega) =0,
	\end{equation}
	and
	\begin{equation}\label{spacederivativesecond}
		\lim_{k\to\infty} \mathbb{E}^s_{\nu,t} \left [ \left | \frac{1}{\tau_k} \int_t^{t+\tau_k} \langle p_t, \Delta (\bar x_r - \bar x_t) + b(\bar x_r,\bar u_r) - b(\bar x_t,\bar u_t) \rangle_{H^1_0(\Lambda)\times H^{-1}(\Lambda)} \mathrm{d}r \right | \right ](\omega) =0,
	\end{equation}
	and
	\begin{equation}\label{spacederivativethird}
		\lim_{k\to\infty} \mathbb{E}^s_{\nu,t} \left [ \left | \frac{1}{\tau_k} \int_t^{t+\tau_k} \langle P^{n}_t(\bar x_r - \bar x_t), \Delta \bar x_r + b(\bar x_r,\bar u_r) \rangle_{H^1_0(\Lambda)\times H^{-1}(\Lambda)} \mathrm{d}r \right | \right ](\omega) =0.
	\end{equation}
	We only consider the terms involving $\Delta \bar x_r$; the terms involving $b(\bar x_r,\bar u_r)$ can be handled similarly.
	
	Let us start with \eqref{spacederivativefirst}. Using \eqref{fprime} and the bound $\| Da(r,\bar x_r) \|_{H^1_0(\Lambda)} \leq 2 \|\bar x_r-\bar x_t\|_{H^1_0(\Lambda)}$, we obtain
	\begin{equation}
		\begin{split}
			&\mathbb{E}^s_{\nu,t} \left [ \left | \frac{1}{\tau_k} \int_t^{t+\tau_k} \langle F^{\prime}(a(r,\bar x_r)) Da(r,\bar x_r), \Delta \bar x_r \rangle_{H^1_0(\Lambda)\times H^{-1}(\Lambda)} \mathrm{d}r \right | \right ](\omega)\\
			&\leq \mathbb{E}^s_{\nu,t} \left [ \frac{1}{\tau_k} \int_t^{t+\tau_k} C(1+ \| \bar x_r \|_{L^2(\Lambda)}^2) \| \bar x_r-\bar x_t \|_{H^1_0(\Lambda)} \| \bar x_r \|_{H^1_0(\Lambda)} \mathrm{d}r \right ](\omega).
		\end{split}
	\end{equation}
	Using H\"older's inequality, we obtain
	\begin{equation}\label{hoelder}
		\begin{split}
			& \mathbb{E}^s_{\nu,t} \left [ \frac{1}{\tau_k} \int_t^{t+\tau_k} \| \bar x_r \|_{L^2(\Lambda)}^2 \| \bar x_r-\bar x_t \|_{H^1_0(\Lambda)} \| \bar x_r \|_{H^1_0(\Lambda)} \mathrm{d}r \right ](\omega)\\
			&\leq \left ( \frac{1}{\tau_k} \int_t^{t+\tau_k} \mathbb{E}^s_{\nu,t} \left [ \| \bar x_r-\bar x_t \|_{H^1_0(\Lambda)}^2 \right ](\omega) \mathrm{d}r \right )^{\frac12} \left ( \frac{1}{\tau_k} \int_t^{t+\tau_k} \mathbb{E}^s_{\nu,t} \left [ \| \bar x_r \|_{L^2(\Lambda)}^4 \| \bar x_r \|_{H^1_0(\Lambda)}^2 \right ](\omega) \mathrm{d}r \right )^{\frac12}.
		\end{split}
	\end{equation}
	Since $\bar x \in L^2([s,T]\times\Omega;H^1_0(\Lambda))$,
	\begin{equation}
		\lim_{k\to\infty} \frac{1}{\tau_k} \int_t^{t+\tau_k} \mathbb{E} \left [ \| \bar x_r-\bar x_t \|_{H^1_0(\Lambda)}^2 \right ] \mathrm{d}r = 0.
	\end{equation}
	Thus, the first factor of \eqref{hoelder} converges to zero $\mathbb{P}$-almost surely along some subsequence. Since
	\begin{equation}
		\| \bar x \|_{L^2(\Lambda)}^{4} \| \bar x \|_{H^1_0(\Lambda)}^2 \in L^1([s,T]\times\Omega),
	\end{equation}
	(see \cite[Lemma 5.1.5]{liu2015}), the second factor is finite in the limit $k\to\infty$ along some subsequence.
	
	Now, let us consider \eqref{spacederivativesecond}. We have
	\begin{equation}
		\begin{split}
			&\mathbb{E}^s_{\nu,t} \left [ \left | \frac{1}{\tau_k} \int_t^{t+\tau_k} \langle p_t, \Delta (\bar x_r - \bar x_t) \rangle_{H^1_0(\Lambda)\times H^{-1}(\Lambda)} \mathrm{d}r \right | \right ](\omega)\\
			&\leq \| p_t \|_{H^1_0(\Lambda)} \left ( \frac{1}{\tau_k} \int_t^{t+\tau_k} \mathbb{E}^s_{\nu,t} \left [ \| \bar x_r - \bar x_t \|_{H^1_0(\Lambda)}^2 \right ](\omega) \mathrm{d}r \right )^{\frac12}.
		\end{split}
	\end{equation}
	The first factor is finite for almost every $t$, and the second factor again converges to zero along some subsequence.
	
	Finally, for \eqref{spacederivativethird}, we obtain using H\"older's inequality
	\begin{equation}\label{spacederivativethird2}
		\begin{split}
			&\mathbb{E}^s_{\nu,t} \left [ \left | \frac{1}{\tau_k} \int_t^{t+\tau_k} \langle P^{n}_t(\bar x_r - \bar x_t), \Delta \bar x_r \rangle_{H^1_0(\Lambda)\times H^{-1}(\Lambda)} \mathrm{d}r \right | \right ](\omega)\\
			&\leq \| P^{n}_t \|_{H^1_0(\Lambda^2)} \left ( \frac{1}{\tau_k} \int_t^{t+\tau_k} \mathbb{E}^s_{\nu,t} \left [ \| \bar x_r - \bar x_t \|_{L^2(\Lambda)}^2 \right ](\omega) \right )^{\frac12}\left ( \frac{1}{\tau_k} \int_t^{t+\tau_k} \mathbb{E}^s_{\nu,t} \left [ \| \bar x_r \|_{H^1_0(\Lambda)}^2 \right ](\omega) \mathrm{d}r \right )^{\frac12}
		\end{split}
	\end{equation}
	The second factor converges to zero and the third factor is finite along some subsequence, which shows that the right-hand side of equation \eqref{spacederivativethird2} converges to zero.
	
	Now, let us consider the second space derivative of $\phi^{n}$. We have
	\begin{equation}
		D^2 \phi^{n}(r,\bar x_r) = F^{\prime\prime}(a(r,\bar x_r)) Da(r,\bar x_r) \otimes Da(r,\bar x_r) + F^{\prime}(a(r,\bar x_r)) D^2a(r,\bar x_r) + P^{n}_t,
	\end{equation}
	thus we have to show
	\begin{equation}\label{secondspacederivativefirst}
		\lim_{k\to\infty} \mathbb{E}^s_{\nu,t}\! \left [ \left | \frac{1}{\tau_k}\! \int_t^{t+\tau_k}\!\!\! \text{tr}\left (\sigma(\bar x_r,\bar u_r)^{\ast} F^{\prime\prime}(a(r,\bar x_r)) Da(r,\bar x_r) \otimes Da(r,\bar x_r) \sigma(\bar x_r,\bar u_r) \right ) \mathrm{d}r \right | \right ]\!(\omega) =0.
	\end{equation}
	and
	\begin{equation}\label{secondspacederivativesecond}
		\lim_{k\to\infty} \mathbb{E}^s_{\nu,t} \left [ \left | \frac{1}{\tau_k} \int_t^{t+\tau_k} \text{tr}\left (\sigma(\bar x_r,\bar u_r)^{\ast} F^{\prime}(a(r,\bar x_r)) D^2a(r,\bar x_r) \sigma(\bar x_r,\bar u_r) \right ) \mathrm{d}r \right | \right ](\omega) =0.
	\end{equation}
	and
	\begin{equation}\label{secondspacederivativethird}
		\lim_{k\to\infty} \mathbb{E}^s_{\nu,t} \left [ \left | \frac{1}{\tau_k} \int_t^{t+\tau_k} \text{tr}\left ( (\sigma(\bar x_r,\bar u_r) \sigma(\bar x_r,\bar u_r)^{\ast} - \sigma(\bar x_t,\bar u_t) \sigma(\bar x_t,\bar u_t)^{\ast} ) P^{n}_t \right ) \mathrm{d}r \right | \right ](\omega) =0.
	\end{equation}
	For the first term, we have
	\begin{equation}
		\begin{split}
			&\mathbb{E}^s_{\nu,t} \left [ \left | \frac{1}{\tau_k} \int_t^{t+\tau_k} \text{tr}\left (\sigma(\bar x_r,\bar u_r)^{\ast} F^{\prime\prime}(a(r,\bar x_r)) Da(r,\bar x_r) \otimes Da(r,\bar x_r) \sigma(\bar x_r,\bar u_r) \right ) \mathrm{d}r \right | \right ](\omega)\\
			&\leq \left ( \frac{1}{\tau_k} \int_t^{t+\tau_k} \mathbb{E}^s_{\nu,t} \left [ \| \sigma(\bar x_r,\bar u_r) \|^2_{L_2(\Xi,L^2(\Lambda))} \right ](\omega) \mathrm{d}r \right )^{\frac12}\\
			&\qquad \left ( \frac{1}{\tau_k} \int_t^{t+\tau_k} \mathbb{E}^s_{\nu,t} \left [ \| F^{\prime\prime}(a(r,\bar x_r)) Da(r,\bar x_r) \otimes Da(r,\bar x_r)  \|^2_{L(L^2(\Lambda))} \right ](\omega) \mathrm{d}r \right )^{\frac12}.
		\end{split}
	\end{equation}
	Since $\sigma(\bar x,\bar u)\in L^2([s,t]\times\Omega;L_2(\Xi,L^2(\Lambda)))$, the first factor is finite along some subsequence, and using continuity of the second derivative of $F$ and the first derivative of $a$, we obtain
	\begin{equation}
		\lim_{k\to\infty} \frac{1}{\tau_k} \int_t^{t+\tau_k} \mathbb{E}^s_{\nu,t} \left [ \| F^{\prime\prime}(a(r,\bar x_r)) Da(r,\bar x_r) \otimes Da(r,\bar x_r)  \|^2_{L(L^2(\Lambda))} \right ](\omega) \mathrm{d}r = 0.
	\end{equation}
	Similar arguments can be employed to prove \eqref{secondspacederivativesecond}. Finally, for \eqref{secondspacederivativethird}, we have
	\begin{equation}
		\begin{split}
			&\mathbb{E}^s_{\nu,t} \left [ \left | \frac{1}{\tau_k} \int_t^{t+\tau_k} \text{tr}\left ( (\sigma(\bar x_r,\bar u_r) \sigma(\bar x_r,\bar u_r)^{\ast} - \sigma(\bar x_t,\bar u_t) \sigma(\bar x_t,\bar u_t)^{\ast} ) P^{n}_t \right ) \mathrm{d}r \right | \right ](\omega)\\
			&\leq \| P^{n}_t \|_{L(L^2(\Lambda))} \frac{1}{\tau_k} \int_t^{t+\tau_k} \mathbb{E}^s_{\nu,t} \left [ \| \sigma(\bar x_r,\bar u_r) \sigma(\bar x_r,\bar u_r)^{\ast} - \sigma(\bar x_t,\bar u_t) \sigma(\bar x_t,\bar u_t)^{\ast} \|_{L_1(L^2(\Lambda))} \right ](\omega) \mathrm{d}r.
		\end{split}
	\end{equation}
	The first factor is finite for almost every $t$, and the second factor converges to zero along some subsequence since $\sigma(x,u)\in L^2([s,T]\times\Omega;L_2(\Xi,L^2(\Lambda)))$.
\end{proof}

\section{Verification Theorem}\label{verificationtheorem}

It turns out that the necessary optimality conditions from Theorem \ref{spacetime} and Corollary \ref{gleqh} are closely related to a non-smooth version of the classical verification theorem, which is a sufficient optimality condition. This relationship was first observed in the finite dimensional case in \cite{zhou1997}, and proven in \cite{gozzi2005,gozzi2010}. We are going to show how to extend this result to the case of controlled SPDEs. The results we are going to prove are mathematically independent of Section \ref{necessaryoptimalityconditions} and in fact hold in a more general setting. In particular, we do not use the adjoint states and therefore we drop the assumption that the coefficients of the control problem are Nemytskii operators and we do not restrict to one-dimensional space domains. Let us therefore first formulate the more general setting.  

Let $\Lambda\subset \mathbb{R}^d$, $d\in\mathbb{N}$, be a bounded domain and fix $T>0$. Consider the following control problem: Minimize
\begin{equation}\label{ch4:costfunctional}
	J(s,x;u) := \mathbb{E} \left [ \int_s^T L(x^u_t,u_t) \mathrm{d}t + H(x^u_T) \right ]
\end{equation}
over $u\in \mathcal{U}_s$ subject to
\begin{equation}\label{ch4:stateequation}
	\begin{cases}
		\mathrm{d}x^u_t = [ \Delta x^u_t + B(x^u_t, u_t) ] \mathrm{d}t + \Sigma(x^u_t,u_t) \mathrm{d}W_t,\quad t\in[s,T]\\
		x^u_s=x\in L^2(\Lambda),
	\end{cases}
\end{equation}
where $L: L^2(\Lambda)\times U \to \mathbb{R}$, $H:L^2(\Lambda)\to\mathbb{R}$ are continuous and satisfy
\begin{equation}\label{growthlh}
	|L(x,u)|, |H(x)| \leq C\left (1+ \| x\|_{L^2(\Lambda)}^2 \right )
\end{equation}
for all $x\in L^2(\Lambda)$, and $B:L^2(\Lambda)\times U \to L^2(\Lambda)$ and $\Sigma : L^2(\Lambda)\times U \to L_2(\Xi,L^2(\Lambda))$ satisfy
\begin{equation}\label{assumptionsbsigma}
	\begin{split}
		\| B(x,u) - B(\tilde x,u) \|_{L^2(\Lambda)} &\leq C \| x-\tilde x\|_{L^2(\Lambda)}\\
		\| B(x,u) \|_{L^2(\Lambda)} &\leq C\left (1+ \| x \|_{L^2(\Lambda)} \right )\\
		\| \Sigma(x,u) - \Sigma(\tilde x,u) \|_{L_2(\Xi,L^2(\Lambda))} &\leq C \| x-\tilde x\|_{L^2(\Lambda)}\\
		\| \Sigma(x,u) \|_{L_2(\Xi,L^2(\Lambda))} &\leq C \left (1+ \| x \|_{L^2(\Lambda)} \right ),
	\end{split}
\end{equation}
for all $x,\tilde x\in L^2(\Lambda)$ and all $u\in U$, and
\begin{equation}
	\| \Sigma(x,u) \|_{L_2(\Xi,H^1_0(\Lambda))} \leq C \left (1+ \| x \|_{H^1_0(\Lambda)} \right )
\end{equation}
for all $x \in H^1_0(\Lambda)$ and all $u\in U$. In this setting, the Hamiltonian $\mathcal{H}: L^2(\Lambda) \times U \times L^2(\Lambda) \times L(L^2(\Lambda)) \to \mathbb{R}$ is given by
\begin{equation}
	\mathcal{H}(x,u,p,P) := L(x,u) + \langle p, B(x,u) \rangle_{L^2(\Lambda)} + \frac12 \text{tr} \left ( \Sigma(x,u)^{\ast} P \Sigma(x,u) \right ).
\end{equation}

\begin{remark}
	In this section, $u^{\ast}$ denotes an arbitrary admissible control and $x^{\ast}$ denotes the corresponding controlled state.
\end{remark}

\begin{theorem}\label{verification}
	Assume there exists a constant $C$ such that for every $t\in [s,T)$ and $x\in H^1_0(\Lambda)$, 
	\begin{equation}\label{growthcondition}
		V(t+\tau, x) - V(t,x) \leq C \left ( 1 + \|x\|_{H^1_0(\Lambda)}^2 \right ) \tau,
	\end{equation}
	and let $V$ be semiconcave uniformly in $t$, i.e. there exists a constant $C\geq 0$ such that for every $t\in (s,T]$ it holds
	\begin{equation}\label{semiconcavity}
		V(t,\cdot) - C\| \cdot \|_{L^2(\Lambda)}^2
	\end{equation}
	is concave on $L^2(\Lambda)$. Suppose further that there are adapted processes
	\begin{equation}
		\begin{cases}
			G\in L^2([s,T]\times\Omega;\mathbb{R})\\
			p\in L^2([s,T]\times\Omega;H^1_0(\Lambda))\\
			P\in L^2([s,T]\times\Omega;L_2(L^2(\Lambda)))
		\end{cases}
	\end{equation}
	such that for almost every $t\in [s,T]$
	\begin{equation}\label{differentialinclusion}
		(G_t,p_t,P_t) \in D^{1,2,+}_{t+,x} V(t, x^{\ast}_t)
	\end{equation}
	$\mathbb{P}$-almost surely, and
	\begin{equation}\label{sufficientassumption}
		\mathbb{E} \left [ \int_s^T G_t + \langle \Delta x^{\ast}_t, p_t\rangle_{H^{-1}(\Lambda)\times H^1_0(\Lambda)} + \mathcal{H}( x^{\ast}_t, u^{\ast}_t,p_t,P_t) \mathrm{d}t \right ] \leq 0.
	\end{equation}
	Then $( x^{\ast}, u^{\ast})$ is an optimal pair.
\end{theorem}

\begin{remark}
	\begin{enumerate}
		\item Together with Corollary \ref{gleqh}, this result implies
		\begin{equation}
			G_t + \langle \Delta x^{\ast}_t, p_t\rangle_{H^{-1}(\Lambda)\times H^1_0(\Lambda)} + \mathcal{H}( x^{\ast}_t, u^{\ast}_t,p_t,P_t) = 0
		\end{equation}
		$\mathrm{d}t\otimes \mathbb{P}$-almost everywhere.
		\item Conditions under which the growth condition \eqref{growthcondition} and the semiconcavity \eqref{semiconcavity} hold are given in the subsequent Proposition \ref{propertiesofvi} and Proposition \ref{propertiesofv}.
	\end{enumerate}
\end{remark}

In the proof of Theorem \ref{verification}, we need the following lemma.

\begin{lemma}\label{fatou}
	Let $V$ satisfy \eqref{growthcondition} and \eqref{semiconcavity}. Then, for almost every $t\in[s,T)$, there is a function $\rho_1 \in L^1(\Omega)$ such that for every $0<\tau\leq T-t$
	\begin{equation}\label{vlipschitzomega}
		\frac{1}{\tau} \mathbb{E}^s_{\nu,t} \left [ V(t+\tau,x^{\ast}_{t+\tau}) - V(t,x^{\ast}_t) \right ] \leq \rho_1(\omega)
	\end{equation}
	$\mathbb{P}$-almost surely. Furthermore, there is a function $\rho_2\in L^1(s,T)$ such that for almost every $t\in [s,T]$ and for every $0<\tau \leq T-t$
	\begin{equation}\label{vlipschitzt}
		\frac{1}{\tau} \mathbb{E} \left [ V(t+\tau,x_{t+\tau}^{\ast}) - V(t,x_t^{\ast}) \right ] \leq \rho_2(t).
	\end{equation}
\end{lemma}

\begin{proof}
	The proof follows along the same lines as in the finite dimensional case, see \cite{gozzi2010}. In order to obtain Lipschitz continuity of the state trajectories in $H^{-1}(\Lambda)$ we rely on the analyticity of the heat semigroup.
	
	First, we split up the increment into separate space and time increments:
	\begin{equation}
		V(t+\tau,x_{t+\tau}^{\ast}) - V(t,x_t^{\ast}) = V(t, x^{\ast}_{t+\tau}) - V(t,x^{\ast}_t) + V(t+\tau,x_{t+\tau}^{\ast}) - V(t, x^{\ast}_{t+\tau}).
	\end{equation}
	For the space increment, we observe that by the semiconcavity of $V(t,\cdot)$, we have for $(G_{t},p_{t},P_{t}) \in D^{1,2,+}_{t+,x}V(t,x^{\ast}_{t})$
	\begin{equation}
		V(t, x^{\ast}_{t+\tau}) - V(t,x^{\ast}_t) \leq \langle p_t, x^{\ast}_{t+\tau}-x^{\ast}_t\rangle_{L^2(\Lambda)} + C \| x^{\ast}_{t+\tau}-x^{\ast}_t \|_{L^2(\Lambda)}^2.
	\end{equation}
	For the time increment, we apply the growth condition \eqref{growthcondition}. Therefore, we obtain altogether
	\begin{equation}\label{estimatev}
		\begin{split}
			&V(t+\tau, x^{\ast}_{t+\tau}) - V(t,x^{\ast}_t)\\
			&\leq \langle p_{t}, x^{\ast}_{t+\tau}-x^{\ast}_t\rangle_{L^2(\Lambda)} + C \| x^{\ast}_{t+\tau}-x^{\ast}_t \|_{L^2(\Lambda)}^2 + C\left (1+\|x^{\ast}_{t+\tau}\|_{H^1_0(\Lambda)}^2 \right )\tau.
		\end{split}
	\end{equation}
	For the first term, we have
	\begin{equation}\label{estimatepx}
		\mathbb{E} \left [ \langle p_t , x^{\ast}_{t+\tau}-x^{\ast}_t \rangle_{L^2(\Lambda)} \right ]\leq \mathbb{E} \left [ \| p_t \|_{H^1_0(\Lambda)}^2 \right ]^{\frac12} \mathbb{E} \left [ \left \| \mathbb{E}^s_{\nu,t} \left [ x^{\ast}_{t+\tau} - x^{\ast}_t \right ] \right \|_{H^{-1}(\Lambda)}^2 \right ]^{\frac12}.
	\end{equation}
	For the second factor, we first note
	\begin{equation}\label{mildsolution}
		\mathbb{E}^s_{\nu,t} \left [ x^{\ast}_{t+\tau} - x^{\ast}_{t} \right ] = (S_{\tau}-I) x^{\ast}_{t} + \mathbb{E}^s_{\nu,t} \left [ \int_{t}^{t+\tau} S_{t+\tau-r} B(x^{\ast}_r,u^{\ast}_r) \mathrm{d}r \right ].
	\end{equation}
	Since the heat semigroup is analytic and $0$ is in the resolvent set of the Laplace operator with Dirichlet boundary conditions, using \cite[Chapter 2, Theorem 6.13]{pazy1983} we obtain for the first term
	\begin{equation}
		\| (S_{\tau} - I) x^{\ast}_t \|_{H^{-1}(\Lambda)} = \| (S_{\tau} - I) \Delta^{-\frac12} x^{\ast}_t \|_{L^2(\Lambda)} \leq C \tau \| x^{\ast}_t \|_{H^1_0(\Lambda)}.
	\end{equation}
	For the second term in \eqref{mildsolution}, we have by \eqref{assumptionsbsigma}
	\begin{equation}
		\begin{split}
			\mathbb{E}^s_{\nu,t} \left [ \int_{t}^{t+\tau} \left \| S_{t+\tau-r} B(x^{\ast}_r,u^{\ast}_r) \right \|_{H^{-1}(\Lambda)} \mathrm{d}r \right ] &\leq C\tau \sup_{r\in [t,t+\tau]} \mathbb{E}^s_{\nu,t} \left [ \| B(x^{\ast}_r,u^{\ast}_r) \|_{L^2(\Lambda)} \right ]\\
			&\leq C \left (1+\|x^{\ast}_{t}\|_{L^2(\Lambda)} \right )\tau.
		\end{split}
	\end{equation}
	Thus, we obtain
	\begin{equation}\label{lipschitztrajectory}
		\left \| \mathbb{E}^s_{\nu,t} \left [ x^{\ast}_{t+\tau} - x^{\ast}_t \right ] \right \|_{L^2(\Lambda)} \leq C\left (1+\|x^{\ast}_t\|_{H^1_0(\Lambda)} \right ) \tau,
	\end{equation}
	and therefore together with \eqref{estimatepx},
	\begin{equation}
		\mathbb{E} \left [ \langle p_t , x^{\ast}_{t+\tau}-x^{\ast}_t \rangle_{L^2(\Lambda)} \right ] \leq C\left (1+ \mathbb{E} \left [ \| x^{\ast}_t\|_{H^1_0(\Lambda)}^2 + \| p_t \|_{H^1_0(\Lambda)}^2 \right ] \right ) \tau.
	\end{equation}
	For the second term in \eqref{estimatev}, we obatin using standard regularity arguments for solutions to SPDEs
	\begin{equation}
		\mathbb{E} \left [ \| x^{\ast}_{t+\tau} - x^{\ast}_t \|_{L^2(\Lambda)}^2 \right ] \leq C \tau.
	\end{equation}
	For the last term in \eqref{estimatev}, we first observe
	\begin{equation}
		\begin{split}
			\| x^{\ast}_{t+\tau} \|_{H^1_0(\Lambda)}^2 &= \| x^{\ast}_{t} \|_{H^1_0(\Lambda)}^2 + \int_t^{t+\tau} \langle \Delta x_r^{\ast} + B(x^{\ast}_r,u^{\ast}_r), x^{\ast}_r \rangle_{H^1_0(\Lambda)} \mathrm{d}r\\
			&\quad + \int_t^{t+\tau} \| \Sigma(x^{\ast}_r,u^{\ast}_r) \|_{L_2(\Xi,H^1_0(\Lambda))}^2 \mathrm{d}r + \int_t^{t+\tau} \langle x^{\ast}_r , \Sigma(x_r^{\ast},u^{\ast}_r) \mathrm{d}W_r \rangle_{H^1_0(\Lambda)}.
		\end{split}
	\end{equation}
	Therefore,
	\begin{equation}
		\begin{split}
			&\mathbb{E} \left [ \| x^{\ast}_{t+\tau} \|_{H^1_0(\Lambda)}^2 \right ]\\
			&\leq \mathbb{E} \left [ \|x^{\ast}_t \|_{H^1_0(\Lambda)}^2 \right ] + \int_t^{t+\tau} \mathbb{E} \left [ \| B(x^{\ast}_r , u^{\ast}_r ) \|_{L^2(\Lambda)}^2 + \| \Sigma(x^{\ast}_r,u^{\ast}_r) \|_{L_2(\Xi,H^1_0(\Lambda))}^2 \right ] \mathrm{d}r.
		\end{split}
	\end{equation}
	Using the growth condition on $B$ and $\Sigma$ and applying Gr\"onwall's inequality, we obtain for the last term in \eqref{estimatev}
	\begin{equation}
		\mathbb{E} \left [ C \left (1+ \| x^{\ast}_{t+\tau} \|_{L^2(\Lambda)}^2 \right ) \tau \right ] \leq C \left ( 1+ \mathbb{E} \left [ \left \| x^{\ast}_t \right \|_{H^1_0(\Lambda)}^2 \right ] \right ) \tau.
	\end{equation}
	Thus, we have the bound
	\begin{equation}
		\frac{1}{\tau} \mathbb{E} \left [ V(t+\tau, x^{\ast}_{t+\tau}) - V(t,x^{\ast}_t) \right ] \leq C\left (1+ \mathbb{E} \left [ \| x^{\ast}_t\|_{H^1_0(\Lambda)}^2 + \| p_t \|_{H^1_0(\Lambda)}^2 \right ] \right ),
	\end{equation}
	where the right-hand side is in $L^1(s,T)$. This proves \eqref{vlipschitzt}. The proof of \eqref{vlipschitzomega} is similar.
\end{proof}

Now, let us prove Theorem \ref{verification}.

\begin{proof}
	Using Proposition \ref{testfunction}, we obtain for given $(t,\omega)\in [s,T]\times\Omega$ a function $\phi\in C^{1,2}((s,T)\times L^2(\Lambda))$ such that $V-\phi$ attains its strict global maximum over $[t,T)\times L^2(\Lambda)$ at the point $(t, x^{\ast}_t(\omega))$, and
	\begin{equation}
		(\phi(t, x^{\ast}_t), \partial_t\phi(t, x^{\ast}_t), D\phi (t, x^{\ast}_t), D^2 \phi(t, x^{\ast}_t)) = (V(t, x^{\ast}_t),G_t,p_t,P_t).
	\end{equation}
	For $\phi$ and $\phi^n$, we use the same construction as in the proof of Corollary \ref{gleqh}.
	
	For $t> T$, we set $V(t, x^{\ast}_t) := V(T,x_T^{\ast})$. Then we have
	\begin{equation}
		\begin{split}
			&\mathbb{E} \left [ V(T,x_T^{\ast}) \right ] - \mathbb{E} \left [ V(s,x_s^{\ast}) \right ]\\
			&= \lim_{\tau\downarrow 0} \frac{1}{\tau} \left ( \int_{T}^{T+\tau} \mathbb{E} \left [ V(t,x_t^{\ast}) \right ] \mathrm{d}t - \int_{s}^{s+\tau} \mathbb{E} \left [ V(t,x_t^{\ast}) \right ] \mathrm{d}t\right )\\
			& = \lim_{\tau\downarrow 0} \frac{1}{\tau} \left ( \int_{s+\tau}^{T+\tau} \mathbb{E} \left [ V(t,x_t^{\ast}) \right ] \mathrm{d}t - \int_{s}^{T} \mathbb{E} \left [ V(t,x_t^{\ast}) \right ] \mathrm{d}t \right )\\
			& = \lim_{\tau\downarrow 0} \int_s^T \frac{1}{\tau} \mathbb{E} \left [ V(t+\tau,x_{t+\tau}^{\ast}) - V(t,x_t^{\ast}) \right ] \mathrm{d}t.
		\end{split}
	\end{equation}
	By Lemma \ref{fatou}, we can apply Fatou's lemma to obtain
	\begin{equation}\label{limit}
		\mathbb{E} \left [ V(T,x_T^{\ast}) \right ] - \mathbb{E} \left [ V(s,x_s^{\ast}) \right ] \leq \int_s^T \limsup_{\tau\downarrow 0} \left ( \frac{1}{\tau} \mathbb{E} \left [ V(t+\tau,x_{t+\tau}^{\ast}) - V(t,x_t^{\ast}) \right ] \right ) \mathrm{d}t.
	\end{equation}
	Since $V-\phi$ attains its maximum at $(t,x^{\ast}_t)$, we have
	\begin{equation}\label{approximationterm}
		\begin{split}
			&\mathbb{E}^s_{\nu,t} \left [ V(t+\tau, x^{\ast}_{t+\tau}) - V(t, x^{\ast}_t) \right ]\\
			&\leq \mathbb{E}^s_{\nu,t} \left [ \phi^n(t+\tau,x^{\ast}_{t+\tau}) - \phi^n(t,x^{\ast}_{t}) \right ]\\
			&\quad + \mathbb{E}^s_{\nu,t} \left [ \phi(t+\tau,x^{\ast}_{t+\tau}) - \phi^n(t+\tau,x^{\ast}_{t+\tau}) + \phi^n(t,x^{\ast}_{t}) - \phi(t,x^{\ast}_{t}) \right ].
		\end{split}
	\end{equation}
	For the last line, we first observe
	\begin{equation}
		\phi^n(\theta,x) - \phi(\theta,x) = \frac12 \langle x- x^{\ast}_t, (P^n_t - P_t) (x- x^{\ast}_t) \rangle_{L^2(\Lambda)},
	\end{equation}
	hence $\phi^n(t, x^{\ast}_t) - \phi(t, x^{\ast}_t)$ vanishes. Furthermore
	\begin{equation}\label{constant}
		\begin{split}
			&| \mathbb{E}^s_{\nu,t} \left [ \phi(t+\tau, x^{\ast}_{t+\tau}) - \phi^n(t+\tau, x^{\ast}_{t+\tau}) \right ] |\\
			& \leq \frac12 \mathbb{E}^s_{\nu,t} \left [ | \langle  x^{\ast}_{t+\tau} - x^{\ast}_t, (P_t-P^n_t)( x^{\ast}_{t+\tau} -  x^{\ast}_t) \rangle_{L^2(\Lambda)} | \right ]\\
			&\leq \frac12 \mathbb{E}^s_{\nu,t} \left [ \|  x^{\ast}_{t+\tau}-   x^{\ast}_t \|_{L^2(\Lambda)}^4 \right ]^{\frac12} \| P_t - P_t^n \|_{L(L^2(\Lambda))}\\
			&\leq C \tau \| P_t - P_t^n \|_{L(L^2(\Lambda))}.
		\end{split}
	\end{equation}
	As in the proof of Lemma \ref{asconvergence}, we fix $\omega\in\Omega$ and apply It\^o's formula to $\phi^n$ on the probability space $(\Omega,\mathcal{F}^{\nu}, \mathbb{P}(\,\cdot\,|\mathcal{F}^s_{\nu,t})(\omega))$. This yields for every $t\in (s,T)$ and every $\tau \geq 0$ such that $t+\tau \leq T$
	\begin{equation}
		\begin{split}
			&\mathbb{E}^s_{\nu,t} \left [ \phi^n(t+\tau,x^{\ast}_{t+\tau}) - \phi^n(t,x^{\ast}_{t}) \right ](\omega)\\
			&= \mathbb{E}^s_{\nu,t} \left [ \int_t^{t+\tau} \partial_{\theta} \phi^n(r, x^{\ast}_r) + \langle D\phi^n(r, x^{\ast}_r), \Delta  x^{\ast}_r + B( x^{\ast}_r, u^{\ast}_r) \rangle \mathrm{d}r \right ](\omega)\\
			&\quad + \mathbb{E}^s_{\nu,t} \left [ \int_t^{t+\tau} \frac12 \text{tr}(\Sigma ( x^{\ast}_r, u^{\ast}_r)^{\ast} D^2\phi^n(r, x^{\ast}_r) \Sigma( x^{\ast}_r, u^{\ast}_r)) \mathrm{d}r\right ](\omega).
		\end{split}
	\end{equation}

	Now, we take the expectation in \eqref{approximationterm}, divide by $\tau$ and take the limit superior $\tau \to 0$. By Lemma \ref{fatou}, we can again apply Fatou's lemma which yields
	\begin{equation}\label{infinitesimal}
		\begin{split}
			&\limsup_{\tau\downarrow 0} \frac{1}{\tau} \mathbb{E} \left [ V(t+\tau, x^{\ast}_{t+\tau}) - V(t, x^{\ast}_{t}) \right ]\\
			&\leq \mathbb{E} \left [ \limsup_{\tau\downarrow 0} \frac{1}{\tau} \mathbb{E}^s_{\nu,t} \left [ V(t+\tau, x^{\ast}_{t+\tau}) - V(t, x^{\ast}_{t}) \right ](\omega) \right ]\\
			&\leq \mathbb{E} \Bigg [ \limsup_{\tau\downarrow 0} \frac{1}{\tau} \mathbb{E}^s_{\nu,t} \Bigg [ \int_t^{t+\tau} \partial_{\theta} \phi^n(r, x^{\ast}_r) + \langle D\phi^n(r, x^{\ast}_r), \Delta  x^{\ast}_r + B( x^{\ast}_r, u^{\ast}_r) \rangle \mathrm{d}r\\
			&\qquad + \int_t^{t+\tau} \frac12 \text{tr}(\Sigma ( x^{\ast}_r, u^{\ast}_r)^{\ast} D^2\phi^n(r, x^{\ast}_r) \Sigma( x^{\ast}_r, u^{\ast}_r)) \mathrm{d}r \Bigg ](\omega)+ C \| P_t - P_t^n \|_{L(L^2(\Lambda))} \Bigg ]
		\end{split}
	\end{equation}
	By Lemma \ref{asconvergence}, this implies
	\begin{equation}
		\begin{split}
			&\limsup_{\tau\downarrow 0} \frac{1}{\tau} \mathbb{E} \left [ V(t+\tau, x^{\ast}_{t+\tau}) - V(t, x^{\ast}_{t}) \right ]\\
			&\leq \mathbb{E} \left [ G_t + \langle p_t, \Delta \bar x_t + B( \bar x_t, \bar u_t) \rangle + \frac12 \text{tr}(\Sigma ( \bar x_t, \bar u_t)^{\ast} P^n_t \Sigma( \bar x_t, \bar u_t)) + C \| P_t - P_t^n \|_{L(L^2(\Lambda))} \right ].
		\end{split}
	\end{equation}
	Together with \eqref{limit}, this yields after taking the limit $n\to\infty$ and plugging in the terminal condition for $V$
	\begin{equation}\label{comparison}
		\mathbb{E} \left [ \int_s^T L( x^{\ast}_t, u^{\ast}_t) \mathrm{d}t + H( x^{\ast}_T) \right ] \leq V(s,x),
	\end{equation}
	which concludes the proof.
\end{proof}

\begin{proposition}\label{propertiesofvi}
	Let $H:L^2(\Lambda)\to \mathbb{R}$ be Lipschitz continuous with respect to the $H^{-1}(\Lambda)$-norm, and twice Fr\'echet differentiable with bounded second derivative. Then there exists a constant $C$ such that for every $t\in [s,T)$ and $x\in H^1_0(\Lambda)$, 
	\begin{equation}
		V(t+\tau, x) - V(t,x) \leq C \left ( 1 + \|x\|_{H^1_0(\Lambda)}^2 \right ) \tau.
	\end{equation}
\end{proposition}

\begin{proof}
	Again, the proof is similar as in the finite dimensional case, see \cite[Section IV.8]{fleming2006}, and the Lipschitz continuity of the state trajectories in $H^{-1}(\Lambda)$ relies on the analyticity of the heat semigroup.
	
	Let $u$ be an admissible control defined on $[t,T]$, and let $x^u$ denote the associated state with initial condition $x^u_t=x$. We introduce a time shifted control $\tilde u$ defined on $[t+\tau,T]$ given by
	\begin{equation}
		\tilde u(r) := u(r-\tau).
	\end{equation}
	Then, we obtain
	\begin{equation}
		\begin{split}
			J(t+\tau,x;\tilde u) &= \mathbb{E} \left [ \int_{t+\tau}^{T} L(x^{\tilde u}_r,\tilde u_r) \mathrm{d}r + H(x^{\tilde u}_{T}) \right ]\\
			&= \mathbb{E} \left [ \int_t^{T-\tau} L(x^u_r,u_r) \mathrm{d}r + H(x^u_{T-\tau}) \right ].
		\end{split}
	\end{equation}
	Hence,
	\begin{equation}
		J(t+\tau,x;\tilde u) - J(t,x,u)	= \mathbb{E} \left [ - \int_{T-\tau}^{T} L(x^u_r,u_r) \mathrm{d}r +  H(x^u_{T-\tau}) - H(x^u_{T}) \right ].
	\end{equation}
	For the running costs, we have by assumption \eqref{growthlh} and standard estimates for solutions of SPDEs
	\begin{equation}\label{runningcosts}
		\mathbb{E} \left [ - \int_{T-\tau}^{T} L(x^u_r,u_r) \mathrm{d}r \right ]\leq C \left (1+ \| x \|_{L^2(\Lambda)}^2 \right ) \tau.
	\end{equation}
	For the terminal costs, using the boundedness of $H_{xx}$ and a Taylor expansion, we obtain
	\begin{equation}\label{terminalcondition}
		\begin{split}
			&\mathbb{E} \left [ H(x^u_{T-\tau}) - H(x^u_{T}) \right ]\\
			&\leq \mathbb{E} \left [ \langle H_x( x^u_{T-\tau}), (x^u_{T-\tau} - x^u_{T}) \rangle_{L^2(\Lambda)} \right ] + C \mathbb{E} \left [ \| x^u_{T-\tau} - x^u_{T} \|_{L^2(\Lambda)}^2 \right ].
		\end{split}
	\end{equation}
	For the first summand, we have
	\begin{equation}
		\begin{split}
			&\mathbb{E} \left [ \langle H_x( x^u_{T-\tau}) ,(x^u_{T-\tau} - x^u_{T}) \rangle_{L^2(\Lambda)} \right ]\\
			&\leq \mathbb{E} \left [ \| H_x(x^u_{T-\tau}) \|_{H^1_0(\Lambda)}^2 \right ]^{\frac12} \mathbb{E} \left [ \left \| \mathbb{E}_{T-\tau} [ x^u_{T} - x^u_{T-\tau} ] \right \|_{H^{-1}(\Lambda)}^2 \right ]^{\frac12}.
		\end{split}
	\end{equation}
	Since $H$ is Lipschitz continuous with respect to the $H^{-1}(\Lambda)$-norm, the Fr\'echet derivative of $H$ maps to $H^1_0(\Lambda)$ and is bounded. For the second factor, we obtain analogous to \eqref{lipschitztrajectory}
	\begin{equation}\label{estimatetimeincrement}
		\mathbb{E} \left [ \left \| \mathbb{E}_{T-\tau} [ x^u_{T} - x^u_{T-\tau} ] \right \|_{H^{-1}(\Lambda)}^2 \right ]^{\frac12} \leq C\left (1+\|x\|_{H^1_0(\Lambda)}^2 \right ) \tau.
	\end{equation}
	For the second summand in \eqref{terminalcondition} we have
	\begin{equation}
		\mathbb{E} \left [ \| x^u_{T-\tau} - x^u_{T} \|_{L^2(\Lambda)}^2 \right ] \leq C \left ( 1+ \|x\|_{L^2(\Lambda)}^2 \right )\tau
	\end{equation}
	by standard regularity results for solutions to SPDEs. Hence, we obtain together with \eqref{runningcosts}
	\begin{equation}
		J(t+\tau,x;\tilde u) - J(t,x,u) \leq C\left (1+\|x\|_{H^1_0(\Lambda)}^2 \right ) \tau.
	\end{equation}
	
	Now, for arbitrary $\varepsilon>0$, let $u^{\varepsilon}$ be a control such that
	\begin{equation}
		J(t,x;u^{\varepsilon}) \leq V(t,x) + \varepsilon.
	\end{equation}
	Then we have
	\begin{equation}
		\begin{split}
			V(t+\tau,x)-V(t,x) &\leq J(t+\tau,x;\tilde u^{\varepsilon}) - J(t,x,u^{\varepsilon}) + \varepsilon\\
			&\leq C \left (1+\|x\|_{H^1_0(\Lambda)}^2 \right ) \tau + \varepsilon.
		\end{split}
	\end{equation}
	Since this holds for arbitrary $\varepsilon$, this concludes the proof.
\end{proof}

\begin{proposition}\label{propertiesofv}
	Let $L,H$ be Lipschitz continuous and semiconcave in $x$ uniformly in $u$. Let $B,\Sigma$ be Fr\'echet differentiable in $x$ with Lipschitz continuous derivative uniformly in $u$. Then the value function $V$ of the control problem \eqref{ch4:costfunctional} and \eqref{ch4:stateequation} is semiconcave uniformly in $t\in [s,T]$, i.e. there exists a $C>0$ such that for every $t\in (s,T]$
	\begin{equation}
		V(t,\cdot) - C \| \cdot \|_{L^2(\Lambda)}^2
	\end{equation}
	is concave.
\end{proposition}

\begin{proof}		
	The corresponding result in the finite dimensional case can be found in \cite[Chapter 4, Proposition 4.5]{yong1999}. The proof in the infinite dimensional case is exactly the same upon replacing the finite dimensional derivatives by Fr\'echet derivatives.
\end{proof}

Theorem \ref{verification} is stated in terms of the value function $V$. However, in most cases it is difficult to obtain a useful representation of the value function directly from its definition. Therefore, instead one uses the Hamilton--Jacobi--Bellman (HJB) equation to characterize the value function. In our setting the HJB equation reads
\begin{equation}\label{ch4:hjb}
	\begin{cases}
		V_s + \langle \Delta x, DV \rangle_{L^2(\Lambda)} + \inf_{u\in U} \mathcal{H}(x,u,DV,D^2 V) = 0\\
		V(T,x) = H(x).
	\end{cases}
\end{equation}
In most interesting cases, the value function is not sufficiently regular to satisfy the HJB equation in a classical sense. Therefore, various generalized notions of solution have been introduced for these kinds of equations, one of which is the notion of $B$-continuous viscosity solutions. Let us recall the definition here. To this end, let us first introduce the relevant class of test functions, see \cite[Definition 3.32]{fabbri2017}.
\begin{definition}\label{testfunctiondefinition}
	A function $\psi$ is a test function if $\psi = \varphi + h(t,\|x\|_{L^2(\Lambda)})$, where
	\begin{enumerate}
		\item $\varphi\in C^{1,2}((0,T)\times L^2(\Lambda))$ is locally bounded, and is such that $\varphi$ is $B$-lower semicontinuous, and $\varphi_t$, $\Delta D\varphi$, $D\varphi$, $D^2\varphi$ are uniformly continuous on $(0,T)\times L^2(\Lambda)$.
		\item $h\in C^{1,2}((0,T)\times \mathbb{R})$ and is such that for every $t\in (0,T)$, $h(t,\cdot)$ is even and $h(t,\cdot)$ is non-decreasing on $[0,\infty)$.
	\end{enumerate}
\end{definition}
Now we recall the definition of a $B$-continuous viscosity solution, see \cite[Definition 3.35]{fabbri2017}.
\begin{definition}
	A locally bounded $B$-upper semicontinuous function $v$ on $(0,T]\times L^2(\Lambda)$ is a viscosity subsolution of \eqref{ch4:hjb} if $v(T,x) \leq H(x)$ for $x\in L^2(\Lambda)$, and whenever $v-\psi$ has a local maximum at a point $(t,x)\in (0,T)\times L^2(\Lambda)$ for a test function $\psi$, then
	\begin{equation}
		\psi_t(t,x) + \langle x, \Delta D\varphi(t,x)\rangle + \inf_{u\in U} \mathcal{H}(x,u,D\psi(t,x),D^2 \psi(t,x)) \geq 0.
	\end{equation}
	A locally bounded $B$-lower semicontinuous function $v$ on $(0,T]\times L^2(\Lambda)$ is a viscosity supersolution of \eqref{ch4:hjb} if $v(T,x) \geq H(x)$ for $x\in L^2(\Lambda)$, and whenever $v+\psi$ has a local minimum at a point $(t,x)\in (0,T)\times L^2(\Lambda)$ for a test function $\psi$, then
	\begin{equation}
		-\psi_t(t,x) - \langle x, \Delta D\varphi(t,x)\rangle + \inf_{u\in U} \mathcal{H}(x,u,-D\psi(t,x),-D^2 \psi(t,x)) \leq 0.
	\end{equation}
	A viscosity solution of \eqref{ch4:hjb} is a function which is both a viscosity subsolution and a viscosity supersolution of \eqref{ch4:hjb}.
\end{definition}

Under additional regularity assumptions, the value function can be characterized as the unique $B$-continuous viscosity solution of the HJB equation \eqref{ch4:hjb}, see \cite[Theorem 3.67]{fabbri2017}. Thus, in combination with the previous result Theorem \ref{verification} we obtain the following verification theorem in terms of a viscosity subsolution to \eqref{ch4:hjb}. 

\begin{theorem}\label{ch4:verification2}
	Suppose there exists a constant $C>0$ such that
	\begin{equation}\label{additionalsigma}
		\| \Sigma(x,u) - \Sigma(\tilde x,u) \|_{L_2(\Xi,L^2(\Lambda))} \leq C \| x-\tilde x \|_{H^{-1}(\Lambda)},
	\end{equation}
	for every $x,\tilde x\in L^2(\Lambda)$ and $u\in U$, and let $L$ and $H$ be locally uniformly continuous in $x\in L^2(\Lambda)$, uniformly in $u\in U$. Let $\mathcal{V}$ be a viscosity subsolution of the HJB equation \eqref{ch4:hjb} satisfying the growth condition \eqref{growthcondition} and the semiconcavity \eqref{semiconcavity}, as well as
	\begin{equation}\label{vequalsh}
		\mathcal{V}(T,x) = H(x)
	\end{equation}
	for all $x\in L^2(\Lambda)$,
	\begin{equation}
		|\mathcal{V}(t,x)| \leq C\left (1+\|x\|_{L^2(\Lambda)}^2 \right ),
	\end{equation}
	and
	\begin{equation}
		\lim_{t\to T} \left ( \mathcal{V}(t,x) - H(S_{T-t} x) \right )^+ =0
	\end{equation}
	uniformly on bounded subsets of $L^2(\Lambda)$, where $(S_r)_{r\geq 0}$ denotes the heat semigroup. Then the following verification theorem holds: Suppose there are adapted processes
	\begin{equation}
		\begin{cases}
			G\in L^2([s,T]\times\Omega;\mathbb{R})\\
			p\in L^2([s,T]\times\Omega;H^1_0(\Lambda))\\
			P\in L^2([s,T]\times\Omega;L_2(L^2(\Lambda)))
		\end{cases}
	\end{equation}
	such that for almost every $t\in [s,T]$
	\begin{equation}
		(G_t,p_t,P_t) \in D^{1,2,+}_{t+,x}\mathcal{V}(t, x^{\ast}_t)
	\end{equation}
	$\mathbb{P}$--almost surely, and
	\begin{equation}
		\mathbb{E} \left [ \int_s^T G_t + \langle \Delta x^{\ast}_t, p_t\rangle_{H^{-1}(\Lambda)\times H^1_0(\Lambda)} + \mathcal{H}( x^{\ast}_t, u^{\ast}_t,p_t,P_t) \mathrm{d}t \right ] \leq 0.
	\end{equation}
	Then $( x^{\ast}, u^{\ast})$ is an optimal pair.
\end{theorem}

\begin{proof}
	Using exactly the same arguments as in the proof of Theorem \ref{verification} up until inequality \eqref{comparison}, and using assumption \eqref{vequalsh}, we obtain
	\begin{equation}
		\mathbb{E} \left [ \int_s^T L( x^{\ast}_t, u^{\ast}_t) \mathrm{d}t + H( x^{\ast}_T) \right ] \leq \mathcal{V}(s,x).
	\end{equation}
	Furthermore, under the additional assumption \eqref{additionalsigma}, the value function is the unique viscosity solution of the HJB equation (see \cite[Theorem 3.67]{fabbri2017} for details). Therefore, applying the comparison result \cite[Theorem 3.54]{fabbri2017}, we have
	\begin{equation}
		\mathcal{V}(s,x) \leq V(s,x)
	\end{equation}
	for any $(s,x) \in (0,T]\times L^2(\Lambda)$, which concludes the proof.
\end{proof}

\begin{remark}\label{remarkiii}
	\begin{enumerate}
		\item In the case of the Nemytskii operator $\sigma$ as defined in Assumption \ref{assumption} (iii), the assumption \eqref{additionalsigma} cannot be satisfied in general. Indeed, consider for example $\Xi = \mathbb{R}$, $(W_t)_{t\in[s,T]}$ is a real-valued Brownian motion, and $\sigma(x,u) = x$, $x\in \mathbb{R}, u \in U$. In this case, assumption \eqref{additionalsigma} reduces to 
		\begin{equation}
			\| x - \tilde x \|_{L^2(\Lambda)} \leq C \| x - \tilde x \|_{H^{-1}(\Lambda)},
		\end{equation}
		for $x,\tilde x\in L^2(\Lambda)$, which is not true. However, one can approximate the Nemytskii operator $\sigma (x(\lambda ),u) (\xi)$ for $x\in L^2 (\Lambda)$, $\lambda\in \Lambda$ and $\xi\in\Xi$, with smooth noise coefficients of the following type 
		\begin{equation}
			\Sigma_\varepsilon (x , u) (\xi) (\lambda)  
			= \sigma \left( \int_\Lambda f_\varepsilon (\lambda ,\mu) x(\mu) \mathrm{d}\mu , u\right) (\xi) 
		\end{equation}
		where  $(f_\varepsilon )_{\varepsilon > 0} \subset C^1 (\Lambda^2)$ is a mollifier such that $\int_\Lambda f_\varepsilon (\lambda ,\mu) x(\mu ) \mathrm{d}\mu \in H_0^1 (\Lambda)$ and the integral converges to $x$ in $L^2(\Lambda)$ as $\varepsilon\to 0$. Under the assumptions imposed on $\sigma$ in Assumption \ref{assumption} (iii), $\Sigma_\varepsilon$ satisfies the additional regularity condition  
		\eqref{additionalsigma}.
		\item All our results generalize to the case of uniformly elliptic operators in divergence form with Dirichlet boundary conditions formally given by
		\begin{equation}
			Ax(\lambda) = \partial_{\lambda} (a \partial_{\lambda}x)(\lambda)
		\end{equation}
		for some $a\in L^{\infty}(\Lambda)$ with $a \geq a_0 > 0$. Indeed, the variational setting relies on the monotonicity and the coercivity of the unbounded operator which also holds for the operator $A$, see also \cite[Remark 6.2]{stannat2020} for more details. Furthermore, $A$ is the generator of an analytic semigroup, hence we can still define $\delta$ as a continuous operator from $H^1_0(\Lambda^2)$ to $L^2(\Lambda)$ (cf. Remark \ref{deltaremark}), and Lemma \ref{regularityterminalcondition} and Proposition \ref{propertiesofv} still hold. Finally, the strong $B$-condition needed for the proof of Theorem \ref{verification} part (i) is shown in \cite[Example 3.16]{fabbri2017}.
		\item In this paper, we chose to work with variational solutions for the state equation \eqref{stateequation}. Other types of solutions such as strong solutions or mild solutions could also be considered. However, for strong solutions, let us recall that in order to assure the existence of a solution one has to impose additional assumptions on the coefficients of the state equation. Once, the existence of a strong solution to the state equation is assured, the proofs significantly simplify since the terms involving the unbounded operator such as $\langle \Delta \bar x_t,p_t \rangle$ in the statement of Theorem \ref{spacetime} and Theorem \ref{verification} are well-defined without any additional regularity for $p_t$, cf. \cite{chen2022}.
		
		On the other hand, for mild solutions, we would like to recall that the mild solution formula does not yield a semimartingale decomposition and therefore there is no (direct) It\^o formula for mild solutions. Moreover, we exploit the fact that the solution of the state equation satisfies
		\begin{equation*}
			\bar x \in L^2([s,T]\times\Omega; H^1_0(\Lambda))
		\end{equation*}
		which is immediate in the variational setting but requires additional assumptions in the case of mild solutions.
	\end{enumerate}	
\end{remark}

\section*{Acknowledgements} The authors would like to express their gratitude to Andrzej {\'{S}}wi{\k{e}}ch for his helpful comments on a first version of this manuscript. Furthermore, the authors would like to thank the anonymous reviewers for their constructive feedback, which helped to improve the manuscript.

\section*{Funding} This work was funded by Deutsche Forschungsgemeinschaft (DFG) through grant CRC 910 ``Control of self-organizing nonlinear systems: Theoretical methods and concepts of application,'' Project (A10) ``Control of stochastic mean-field equations with applications to brain networks,'' while the second author was affiliated with Technische Universit\"at Berlin.


\end{document}